\documentclass[a4paper,12pt]{article}

\usepackage{mymacros}

\newcommand{\nsubset}{\not \subset}

\newcommand{\perfmat}{\operatorname{Perf}}

\newcommand{\ahilb}{\mathop{A\text{-}{\mathrm{Hilb}}}}

\title{Dimer models and the special McKay correspondence}
\author{Akira Ishii and Kazushi Ueda}
\date{}

\begin{document}

\maketitle

\begin{abstract}
We study the behavior of a dimer model
under the operation of removing a corner from the lattice polygon
and taking the convex hull of the rest.
This refines an operation of Gulotta,
and the special McKay correspondence plays an essential role
in this refinement.
As a corollary, we show that for any lattice polygon,
there is a dimer model such that
the derived category of finitely-generated modules
over the path algebra of the corresponding quiver with relations
is equivalent to the derived category of coherent sheaves
on a toric Calabi-Yau 3-fold determined by the lattice polygon.
Our proof is based on a detailed study of relationship
between combinatorics of dimer models and
geometry of moduli spaces,
and does not depend on the result
of Bridgeland, King and Reid
\cite{Bridgeland-King-Reid}.
\end{abstract}

\section{Introduction}

Dimer models are introduced in 1960s
as statistical mechanical models
which include the two-dimensional Ising model
as a special case.
See e.g. \cite{Baxter_ESMSM, Kenyon_IDM}
and references therein
for more on this aspect of dimer models.
In this paper, a {\em dimer model} is
a bicolored graph on a real 2-torus
giving a polygon division of the torus.
A fundamental object associated with a dimer model
from statistical mechanical point of view
is its {\em characteristic polynomial}.
It is a Laurent polynomial in two variables
defined in purely combinatorial way
in terms of {\em perfect matchings}.
The Newton polygon of the characteristic polynomial
is called the {\em characteristic polygon}.

More recently,
string theorists has discovered that
dimer models encode the information of quivers with relations,
and used them to study
supersymmetric quiver gauge theories
in four dimensions
(see e.g. \cite{Kennaway_BT} and references therein).
If a dimer model is {\em non-degenerate},
then the moduli space $\scM_{\theta}$
of stable representations of the corresponding quiver
with dimension vector $(1, \dots, 1)$
with respect to a generic stability parameter $\theta$
in the sense of King \cite{King}
is a smooth toric Calabi-Yau 3-fold
\cite{Ishii-Ueda_08}.
Here, a stability parameter is {\em generic}
if any semi-stable objects are stable.
The Calabi-Yau property of $\scM_\theta$ implies that
the convex hull $\Delta$ of the set of primitive generators
of one-dimensional cones of the fan
describing $\scM_\theta$ as a toric manifold
is a lattice polygon
(i.e. they all lie on a hyperplane).
Moreover, this lattice polygon is known to coincide
with the characteristic polygon
of the dimer model
\cite{Franco-Vegh_MSGTDM, Ishii-Ueda_08}.
Although the structure of the fan is not determined
by this lattice polygon,
any fan structures give equivalent
derived categories of coherent sheaves
\cite{Bondal-Orlov_semiorthogonal, Bridgeland_FDC}.

The quiver associated with a dimer model
is the dual graph of the dimer model,
oriented in such a way that a white node is
on the right of an arrow.
A face of the dimer model
gives a vertex $v$ of the quiver,
which in turn gives
the corresponding tautological line bundle $\scL_v$
on the moduli space $\scM_\theta$.
An edge of the dimer model gives
an arrow $v \to w$ of the quiver,
which corresponds to a morphism $\scL_v \to \scL_w$
of tautological bundles
by the universal morphism
$
 \bC \Gamma \to \End \lb \bigoplus_v \scL_v \rb.
$
These correspondences are summarized
in Table \ref{tb:dictionary1}.

\begin{table}[t]
\centering
\begin{tabular}{c|c|c}
dimer model & quiver & moduli space \\
\hline
face & vertex & tautological line bundle \\
edge & arrow & morphism between tautological line bundles
\end{tabular}
\caption{The correspondence among dimer models, quivers
and moduli spaces}
\label{tb:dictionary1}
\end{table}

Consider the following two conditions:
\begin{itemize}
 \item[(\bfT)]
The tautological bundle $\bigoplus_v \scL_v$
on the moduli space $\scM_\theta$ is a tilting object.
 \item[(\bfE)]
The universal morphism
$
 \bC \Gamma \to \End \lb \bigoplus_v \scL_v \rb
$
is an isomorphism.
\end{itemize}
According to Morita theory for derived category
\cite{Bondal_RAACS, Rickard},
the conditions (\bfT)+(\bfE) imply that the functor
\begin{equation} \label{eq:der_equiv}
 \Phi(-)
   = \bR \Gamma \lb \lb \bigoplus_v \scL_v \rb \otimes - \rb
   : D^b \coh \scM_\theta \to D^b \module \bC \Gamma
\end{equation}
is an equivalence of triangulated categories.

There is a notion of {\em consistency condition}
on a dimer model
\cite{Hanany-Vegh,
Ishii-Ueda_CCDM,
Bocklandt_CCDM},
which ensures the Calabi-Yau property
of the path algebra $\bC \Gamma$
of the quiver with relations
associated with the dimer model
\cite{Mozgovoy-Reineke, Davison, Broomhead}.
An example of a consistent dimer model
comes from a finite abelian subgroup
$A$ of $\SL(3, \bC)$,
where the associated quiver is the McKay quiver.

The Calabi-Yau property of the path algebra $\bC \Gamma$
implies that $\scM_\theta$ is smooth
with the trivial canonical bundle,
and the tautological bundle satisfies
conditions (\bfT)+(\bfE)
by a result of
Bridgeland, King and Reid
\cite{Bridgeland-King-Reid,
Van_den_Bergh_NCR}.
In this paper,
we do not rely on their results and
give an independent proof of these facts
for any consistent dimer model.

%

A {\em corner} of a lattice polygon $\Delta$ is
an extremal point of $\Delta$,
and a {\em side} of $\Delta$ is the interval
between two neighboring corners.
A side is divided into 
{\em primitive side segments},
defined as intervals between adjacent lattice points
on the boundary of $\Delta$.
We reserve the words {\em edge} and {\em vertex}
for an edge of a dimer model
and a vertex of a quiver respectively.

Let $\frakc$ be a corner
of a lattice polygon $\Delta$, and
$\Delta'$ be the convex hull of
the set of lattice points of $\Delta$
other than $\frakc$.
If $\Delta'$ is a lattice polygon
(i.e. if not all lattice points of $\Delta$
other than $\frakc$ lie on a line),
then we say that
the lattice polygon $\Delta'$ is obtained
from the lattice polygon $\Delta$
by removing the corner $\frakc$.

In this paper,
we study the behavior of a dimer model
under the removal of a corner
from the characteristic polygon:

\begin{theorem} \label{th:removal}
Let $G$ be a consistent dimer model
and $\Delta$ be the characteristic polygon of $G$.
Let further $\frakc$ be a corner of $\Delta$ and
$\Delta'$ be the lattice polygon
obtained from $\Delta$ by removing the corner $\frakc$.
Then there is an explicit algorithm
to remove some of the edges from $G$
and produce another dimer model $G'$
satisfying the following two conditions:
\begin{enumerate}
 \item
$G'$ is consistent.
 \item
The characteristic polygon of $G'$ coincides with $\Delta'$.
\end{enumerate}
\end{theorem}

This refines an operation of Gulotta
\cite{Gulotta}
who studied the operation of removing a triangle
from the characteristic polygon.
Since 
\begin{itemize}
 \item
any polygon can be embedded
into a sufficiently large triangle, and
 \item
the McKay quiver gives a consistent dimer model
for any triangle,
\end{itemize}
Theorem \ref{th:removal} gives a constructive proof
of the following:

\begin{corollary} \label{cr:existence}
For any lattice polygon $\Delta$,
there is a consistent dimer model
whose characteristic polygon coincides with $\Delta$.
\end{corollary}

\noindent
Corollary \ref{cr:existence} also follows
from a result of Gulotta
\cite[Theorem 6.1]{Gulotta}
which produces a properly-ordered dimer model
for any lattice polygon,
and a result in \cite[Theorem 1.1]{Ishii-Ueda_CCDM}
which shows that
properly-ordered dimer models are consistent.

Although the algorithm in Theorem \ref{th:removal}
can be stated in a purely combinatorial way,
its motivation comes from geometry of moduli spaces,
where the special McKay correspondence
by Wunram \cite{Wunram}
plays an essential role.

Let $A$ be a finite small subgroup of $\GL_2(\bC)$
and $\ahilb(\bC^2)$ be the Hilbert scheme
of $A$-orbit in $\bC^2$ \cite{Nakamura_HSAGO}.
The Hilbert-Chow morphism
$$
 \pi : \ahilb(\bC^2) \to \bC^2/A = \Spec \bC[x, y]^A
$$
gives the minimal resolution of the quotient singularity
\cite{Ishii_MKG}.
The special McKay correspondence gives
a description of the derived category of coherent sheaves
on $\ahilb(\bC^2)$
in terms of $A$
\cite{Van_den_Bergh_TFNR,
Craw_SMC,
Wemyss_GL2}.

Let $G$ be a consistent dimer model,
$\Delta$ be its characteristic polygon,
and $G'$ be another consistent dimer model
obtained from $G$ by removing a corner $\frakc$ from $\Delta$
as in Theorem \ref{th:removal}.
Let further $\scM_\theta$ be the moduli space of
the quiver $\Gamma$ with relations
associated with the consistent dimer model $G$
and a generic stability parameter $\theta$.
Since $\scM_\theta$ is a smooth toric variety
and $\Delta$ is the convex hull
of primitive generators of one-dimensional cones
of the corresponding fan,
any lattice point of $\Delta$ corresponds
to a divisor in $\scM_\theta$.
A toric divisor $D_\frakc$ of $\scM_\theta$
corresponding to a corner $\frakc$ of $\Delta$
will be called a {\em corner toric divisor}.

\begin{proposition} \label{prop:corner}
Let $G$ be a consistent dimer model and
$\frakc$ be a corner of the characteristic polygon $\Delta$.
Then there is a generic stability parameter $\theta$
and a finite small abelian subgroup $A$
of $\GL_2(\bC)$
satisfying the following:
\begin{itemize}
 \item
There is an open neighborhood
$U_\frakc$
of the corner toric divisor $D_\frakc$ in $\scM_\theta$ and
a commutative diagram
$$
\begin{CD}
 D_\frakc @>>> U_\frakc \\
 @VVV @VV{\varphi}V \\
 \ahilb(\bC^2) @>>> \ahilb(\bC^3)
\end{CD}
$$
where horizontal arrows are closed embeddings and
vertical arrows are isomorphisms.
 \item
For any irreducible representation $\rho$ of $A$,
there is a vertex $v$ of the quiver $\Gamma$
such that the pull-back of the tautological bundles
$\scL_\rho$ on $\ahilb(\bC^3)$ is isomorphic
to the restriction of
$\scL_v$ on $\scM_\theta$;
$$
 \varphi^* \scL_\rho \cong \scL_v|_{U_\frakc}.
$$
\end{itemize}
Here $A \subset \GL_2(\bC)$ is embedded
into $\SL_3(\bC)$ in a natural way.
\end{proposition}

To prove Proposition \ref{prop:corner},
we introduce the notion of
{\em large hexagons}.
A large hexagon is the union of faces of a dimer model,
which is cut out by a pair of zigzag paths.
The tautological line bundles corresponding
to faces of one large hexagon
are isomorphic near the given corner divisor.
A division of a dimer model into large hexagons
gives a coarse graining of the associated quiver
into the McKay quiver for some $A \subset \GL_2(\bC)$.
The correspondence
between combinatorics of dimer models and
geometry of moduli spaces is summarized
in Table \ref{tb:dictionary2}.

\begin{table}[t]
\centering
\begin{tabular}{c|c|c}
dimer model & characteristic polygon & moduli space \\
\hline
perfect matching & lattice point & toric divisor \\
zigzag path & primitive side segment & non-compact torus-invariant curve
\end{tabular}
\caption{The correspondence among
dimer models,
characteristic polygons
and moduli spaces}
\label{tb:dictionary2}
\end{table}

The main result in this paper is the following:

\begin{theorem} \label{th:main}
Let $G$ be a consistent dimer model.
Then for any generic stability parameter $\theta$,
the tautological bundle
$\bigoplus_v \scL_v$ on the moduli space $\scM_\theta$
satisfies the conditions \textup{(\bfT)+(\bfE)}.
\end{theorem}

Theorem \ref{th:main} contains the abelian case
of the main result of Bridgeland, King and Reid
\cite{Bridgeland-King-Reid}.
Our proof is independent of theirs,
and based on Theorem \ref{th:induction} below.

Let $G$ be a consistent dimer model and
$G'$ be another consistent dimer model
obtained from $G$ by removing a corner
from the characteristic polygon as in Theorem \ref{th:removal}.
Choose a stability parameter $\theta$ for $G$
described in Proposition \ref{prop:corner}.
This stability parameter $\theta$ for $G$ naturally induces
a stability parameter $\theta'$ for $G'$, and
let $\scM'_{\theta'}$ be the corresponding moduli space
associated with the dimer model $G'$.
Then $\scM'_{\theta'}$ is naturally
an open subscheme of $\scM_\theta$,
and the complement is exactly the divisor $D_\frakc$;
$$
 \scM'_{\theta'} = \scM_\theta \setminus D_\frakc.
$$
A key to the proof of Theorem \ref{th:main}
is the following:

\begin{theorem} \label{th:induction}
The conditions \textup{(\bfT)+(\bfE)} hold for $\scM_\theta$
if and only if they hold for $\scM'_{\theta'}$.
\end{theorem}

Theorem \ref{th:main} follows
from Theorem \ref{th:induction}
by induction on the number of lattice points
of the characteristic polygon.

The proof of Theorem \ref{th:induction} is based
on a detailed study of the interplay
between combinatorics of dimer models and
geometry of moduli spaces.
The proof also gives the following characterization
of the edges removed in the operation
in Theorem \ref{th:removal},
which explains the geometric origin
of the algorithm:

\begin{proposition} \label{prop:restriction}
The edges removed from $G$
in the operation in Theorem \ref{th:removal}
are exactly those which correspond
to morphisms between tautological bundles
vanishing only on the toric divisor
$D_\frakc \subset \scM_\theta$.
\end{proposition}

The effect of the operation
in Theorem \ref{th:removal}
on various objects is summarized
in Table \ref{tb:removal}.

\begin{table}[t]
\centering
\begin{tabular}{c|c}
object & operation \\
\hline
characteristic polygon &
removing a corner $\frakc \in \Delta$ \\
moduli space &
removing the toric divisor $D_\frakc \subset \scM_\theta$ \\
path algebra &
inverting the arrows
vanishing only on $D_\frakc$ \\
quiver &
contracting the arrows as above \\
dimer model &
removing the edges
dual to the arrows as above
\end{tabular}
\caption{The effect of the operation
in Theorem \ref{th:removal}}
\label{tb:removal}
\end{table}


This paper is organized as follows:
In Section \ref{sc:special-McKay},
we recall the special McKay correspondence
for finite small subgroups of $\GL_2(\bC)$.
In Section \ref{sc:continued-fraction},
we recall the description
of geometry of the minimal resolution of $\bC^2/A$
in terms of continued fraction expansion,
and collect lemmas which will be useful later.
%
In Section \ref{sc:definitions},
we collect basic definitions on dimer models
and associated quivers.
%
%
In Section \ref{sc:consistency},
we recall consistency conditions on dimer models.
In Section \ref{sc:large-hexagon},
we introduce the notion of large hexagons,
which will be our main technical tool.
This will allow us to embed $\ahilb(\bC^3)$
for a suitable $A \subset \GL_2(\bC) \subset \SL_3(\bC)$
into our moduli space.
In Section \ref{sc:non-degeneracy},
we prove that consistent dimer models are non-degenerate.
In Section \ref{sc:corner},
we give a characterization of corner perfect matchings.
In Section \ref{sc:algorithm},
we give an explicit description
of the operation in Theorem \ref{th:removal}.
In Section \ref{sc:preservation-of-consistency},
we prove that
the operation in Theorem \ref{th:removal}
preserves the consistency condition.
In Section \ref{sc:zigzag-polygon},
we show that
the lattice polygon changes as expected
under the operation in Theorem \ref{th:removal}.
This concludes the proof of Theorem \ref{th:removal}.
In Section \ref{sc:polygon},
we prove
Proposition \ref{prop:restriction}.
In Section \ref{sc:injectivity},
we show that
the consistency condition implies
the injectivity of the universal morphism
in condition (\bfE).
Theorem \ref{th:induction}
is proved in Sections \ref{sc:tilting-G-Hilb}--\ref{sc:surj-general},
and
Theorem \ref{th:main} is proved
in Section \ref{sc:der_equiv}.




{\bf Acknowledgment}:
We thank Alastair Craw for
valuable discussions and
suggesting a number of improvements.
We also thank
Nathan Broomhead,
Ben Davison,
Dominic Joyce,
Alastair King,
Diane Maclagan,
Balazs Szendroi,
Yukinobu Toda,
Michael Wemyss and
Masahito Yamazaki
for valuable discussions.
A.~I. is supported by Grant-in-Aid for Scientific Research (No.18540034 and No.21540039).
K.~U. is supported by Grant-in-Aid for Young Scientists (No.18840029 and No.20740037).
A large part of this work has been done
while K.~U. is visiting the University of Oxford,
and he thanks the Mathematical Institute for hospitality
and Engineering and Physical Sciences Research Council
for financial support.

\section{The special McKay correspondence}
 \label{sc:special-McKay}

Let $R := S^A$ be the invariant ring
of the polynomial ring $S = \bC[x_1, \ldots, x_n]$
with respect to the natural action of
a finite small subgroup $A$ of $\GL_n(\bC)$.
For any irreducible representation $\rho$ of $A$,
the invariant part
$
 M_\rho := (S \otimes \rho^\vee)^A
$
is an indecomposable Cohen-Macaulay
(and hence reflexive) $R$-module,
since it a direct summand of a Cohen-Macaulay $R$-module
$S \otimes \rho^\vee$.

The {\em McKay quiver} $\Lambda$ of $A$ is a quiver with relations
whose set of vertices is the set $\Irrep(A)$ of irreducible representations of $A$.
The number $a_{\nu \mu}$ of arrows from a vertex $\mu \in \Irrep(A)$
to another vertex $\nu \in \Irrep(A)$ is given by the multiplicity
in the irreducible decomposition of the tensor product
$$
 \mu \otimes \rho_{\mathrm{Nat}}^\vee
  = \bigoplus_{\nu \in \Irrep(A)} \nu^{\oplus a_{\nu \mu}},
$$
where $\rho_{\mathrm{Nat}} : A \hookrightarrow \GL_n(\bC)$
is the natural representation of $A$
and $(-)^\vee$ denotes the dual representation.
The relations of $\Lambda$ are such that
the path algebra $\bC \Lambda$ is isomorphic to
$\End_R \lb \bigoplus_{\rho \in \Irrep(A)} M_\rho \rb$,
which is Morita equivalent to
$$
 \End_R(S)
  \cong \End_R \lb \bigoplus_{\rho \in \Irrep(A)} M_\rho^{\oplus \dim \rho} \rb
  \cong S \rtimes A.
$$

Now assume that $A$ is a finite small subgroup of $\GL_2(\bC)$,
and let $Y = \ahilb(\bC^2)$ be the Hilbert scheme
of $A$-orbits in $\bC^2$ \cite{Nakamura_HSAGO}.
The Hilbert-Chow morphism
$$
 \pi : Y \to X = \Spec \bC[x, y]^A
$$
gives the minimal resolution of the quotient singularity
\cite{Ishii_MKG}.

\begin{definition-lemma}[Esnault \cite{Esnault_RMQSS}]
Let $\scM$ be a sheaf on $Y$
and $\scM^{\vee}$ be its dual sheaf.
Then there exists a reflexive module $M$ on $X$ such that
$
 \scM \cong \Mtilde := \pi^*M / \text{\it torsion}
$
if and only if the following three conditions are satisfied:
\begin{enumerate}
 \item
$\scM$ is locally-free.
 \item
$\scM$ is generated by global sections.
\item
$H^1((\scM)^{\vee} \otimes \omega_{Y})=0$.
\end{enumerate}
In this case $\scM$ is said to be {\em full}.
\end{definition-lemma}



Let us recall the definition of a tilting object:

\begin{definition} \label{df:tilting}
An object $\scE$ in a triangulated category $\scT$
is {\em acyclic} if
$$
 \Ext^k(\scE, \scE) = 0, \qquad k \ne 0.
$$
It is a {\em generator} if for any object $\scF$,
$$
 \Ext^k(\scE, \scF) = 0
$$
for any $k \in \bZ$ implies $\scF \cong 0$.
An acyclic generator is called a {\em tilting object}.
\end{definition}

A tilting object induces a derived equivalence:

\begin{theorem}[Bondal \cite{Bondal_RAACS}, Rickard \cite{Rickard}]
Let $\scE$ be a tilting object
in the derived category $D^b \coh X$
of coherent sheaves on a smooth quasi-projective variety $X$.
Then $D^b \coh X$ is equivalent to the derived category of
finitely-generated modules
over the endomorphism algebra $\Hom(\scE, \scE)$.
\end{theorem}

The following theorem is the McKay correspondence
as a derived equivalence for a finite subgroup of $\SL_2(\bC)$:

\begin{theorem}[{Kapranov and Vasserot \cite{Kapranov-Vasserot},
see also Bridgeland, King and Reid \cite{Bridgeland-King-Reid}}]
When $A$ is a finite subgroup of $\SL_2(\bC)$,
the direct sum of indecomposable full sheaves
is a tilting object
whose endomorphism ring is Morita equivalent
to the crossed product algebra $\bC[x, y] \rtimes A$.
\end{theorem}



This is no longer true when $A \nsubset \SL_2(\bC)$,
and one has to restrict the class of full sheaves.
The following theorem is due to Wunram:

\begin{theorem}[{Wunram \cite[Main Result]{Wunram}}]
 \label{th:Wunram-1}
Let
$
 C = \bigcup_{i=1}^r C_i
$
be the decomposition
of the exceptional set $C$
into irreducible components.
Then for every curve $C_i$
there exists exactly one indecomposable reflexive module
$M_i$
such that the corresponding full sheaf
$\Mtilde_i = \pi^* M_i / \text{\it torsion}$
satisfies the conditions $H^1((\Mtilde)^{\vee})=0$
and
$$
 c_1(\Mtilde_i) \cdot C_j = \delta_{ij}.
$$
\end{theorem}

A full sheaf is said to be {\em special}
if there is an index $1 \le i \le r$ such that $\scM = \scM_i$
or it is isomorphic to the structure sheaf $\scO_Y$.
The special full sheaf $\scO_Y$ corresponds to the trivial representation
and is denoted by $\scM_0$.
Special full sheaves are characterized as follows:

\begin{theorem}[{Wunram \cite[Theorem 1.2]{Wunram}}]
 \label{th:Wunram-2}
An indecomposable full sheaf $\scM$ is special
if and only if $H^1(\scM^{\vee})=0$.
\end{theorem}

An irreducible representation $\rho$ of $A$ is said to be special
if the corresponding full sheaf
$
 \scM_\rho
  = \pi^*
     \left(
      (\rho^\vee \otimes \bC[x, y])^A
     \right) / \text{torsion}
$
is special.
%

Special full sheaves generate the derived category
of coherent sheaves on $Y$:

\begin{theorem}[{Van den Bergh \cite[Theorem B]{Van_den_Bergh_TFNR}}]
 \label{th:VdB}
The direct sum
of indecomposable special full sheaves
is a tilting object.
\end{theorem}

Let $\scM$ be the direct sum
of indecomposable special full sheaves.
It follows that the derived category $D^b \coh Y$
of coherent sheaves on $Y$ is equivalent
to the derived category
$D^b \module (\End \scM)$
of finitely-generated right modules
over $\End \scM$.
The special McKay correspondence
as a derived equivalence
is studied by Craw \cite{Craw_SMC} and Wemyss \cite{Wemyss_GL2}.
The category $D^b \coh Y \cong D^b \module (\End \scM)$
is an admissible subcategory of
$
 D^b \coh [\bC^2/A]
   \cong D^b \module (\bC[x,y] \rtimes A),
$
whose semiorthogonal complement is generated
by an exceptional collection
\cite{Ishii-Ueda_SMEC}.

\section{Specials and continued fractions}
 \label{sc:continued-fraction}

For relatively prime integers $0<q<n$,
consider the small cyclic subgroup
$
 A = \langle \frac{1}{n}(1, q) \rangle
$
of $\GL_2(\bC)$ generated by
$$
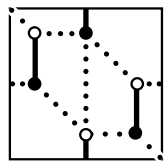
 \frac{1}{n}(1, q)
  = \begin{pmatrix}
     \zeta & 0 \\ 0 & \zeta^q
    \end{pmatrix},
$$
where $\zeta$ is a primitive $n$-th root of unity.
We label the irreducible representations of $A$
by elements $a \in \bZ/n\bZ$ so that
$a$ sends the above generator to $\zeta^{-a}$.
\begin{remark}
$\scM_{\rho}$ in our notation corresponds to $\rho^{\vee}$ by the correspondence in \cite{Wunram}.
So we dualize the labeling of the irreducible representations so that Theorem \ref{theorem:wunram} is of the same form.
\end{remark}

Define integers $r$, $b_1, \dots, b_r$ and $i_0, \dots, i_{r+1}$ as follows:
Put $i_0 := n$, $i_1:=q$ and
define $i_{t+2}$ and $b_{t+1}$ inductively by
\begin{equation} \label{eq:cont-frac}
 i_t = b_{t+1} i_{t+1} - i_{t+2} \quad (0 < i_{t+2} < i_{t+1})
\end{equation}
until we finally obtain $i_r=1$ and $i_{r+1}=0$.
This gives a continued fraction expansion
\begin{equation} \label{eq:continued_fraction}
 \frac{n}{q}
  = b_1 - \cfrac{1}{b_2 - \cfrac{1}{\ddots -\cfrac{1}{b_r}}}
\end{equation}
and $-b_t$ is the self intersection number
of the $t$-th irreducible exceptional curve $C_t$.

For a general representation $d$,
the degrees of the full sheaf $\scL_d$ are given
in the following way:
\begin{theorem}[{Wunram \cite[Theorem]{wunram2}}]
 \label{theorem:wunram}
For an integer $d$ with $0 \le d <n$,
there is a unique expression
$$
d = d_1 i_1 + d_2 i_2 + \dots + d_r i_r
$$
where $d_i \in \bZ_{\ge 0}$ are non-negative integers satisfying
$$
 0 \le \sum_{t>t_0} d_t i_t < i_{t_0}
$$
for any $t_0$.
Then one has
$$
 \deg \scM_d |_{C_t} = d_t
$$
for any $t = 1, \ldots, r$.
\end{theorem}

\begin{remark}
Non-negative integers $d_i$ in Theorem \ref{theorem:wunram}
can be computed
by setting $e_0 = d$ and
$$
 e_t = d_{t+1} i_{t+1} + e_{t+1}, \qquad 0 \le e_{t+1} < i_{t+1}
$$
for $t = 0, \ldots, r-1$.
\end{remark}

\begin{corollary}
Special representations are given by $i_0\equiv i_{r+1}, i_1, \dots, i_r$,
and the labeling of specials and irreducible components are related by
$$
\deg \scM_{i_s}|_{C_t} = \delta_{st}.
$$
\end{corollary}

\begin{lemma}[{Wunram \cite[Lemma 1]{wunram2}}]\label{lemma:wunramvanishing}
A sequence
$
 (d_1, \dots, d_r) \in (\bZ_{\ge 0})^r
$
is obtained from an integer $d \in [0, n-1]$ as in the previous theorem
if and only if the following hold:
\begin{itemize}
 \item $0 \le d_t \le b_t-1$ for any $t$.
 \item If $d_s= b_s-1$ and $d_t=b_t-1$ for $s<t$,
       then there is $l$ with $s<l<t$ and $d_l \le b_l-3$.
\end{itemize}
\end{lemma}


Introduce the dual sequence $j_0, \dots, j_{r+1}$ by
$j_0=0$, $j_1=1$, and 
$$
 j_t = j_{t-1} b_{t-1} - j_{t-2}, \qquad t \ge 2.
$$
Then one has $j_{r+1}=n$.

\begin{lemma}[{Wunram \cite[Lemma 2]{wunram2}}]\label{lm:wunram_dual}
Let $d = d_1 i_1 + \dots + d_r i_r$ be as in Theorem \ref{theorem:wunram}
and put $f=d_1 j_1 + \dots + d_r j_r$. Then one has $qf \equiv d \mod n$.
\end{lemma}

In particular, special representations are given by
\begin{equation}\label{equation:special}
 i_0 \equiv qj_0, \quad
 i_1 \equiv qj_1, \quad
 \dots, \quad
 i_r \equiv qj_r.
\end{equation}
Note that $( i_t )_{t=0}^r$ is decreasing and
$( j_t )_{t=0}^r$ is increasing.

\section{Dimer models and quivers}
 \label{sc:definitions}

\subsection{Dimer models}

By a  {\em graph},
we mean an abstract,
unoriented graph,
possibly with multiple edges and loops.
To be more precise,
a graph is a triple
$(N, E, \partial)$
consisting of
\begin{itemize}
 \item
a set $N$ of nodes,
 \item
a set $E$ of edges, and
 \item
the incidence relation $\partial : E \to N^{(2)}$,
which is a map from $E$
to the symmetric product $N^{(2)} = N^2 / \frakS_2$.
\end{itemize}
%
A graph 
is {\em bipartite}
if one can divide the set
$N$ of nodes into the disjoint union
of
\begin{itemize}
\item
a set $B \subset N$ of black nodes, and
\item
a set $W \subset N$ of white nodes, so that
\item
no edge connects nodes with the same color.
\end{itemize}
A {\em bicolored graph} is a bipartite graph
with a fixed choice of a coloring.

With a graph $(N, E, \partial)$,
one can associate a one-dimensional CW complex
whose 0-cells and 1-cells correspond to nodes and edges
respectively.
An {\em embedding} of a graph into a topological space
$T$ is a continuous injection from this CW complex
to $T$.
When a graph is embedded in a topological space,
we often identify
nodes and edges with their images
under the embedding.

Let $T$ be a real two-torus.
We fix an identification $T = \bR^2 / \bZ^2$,
which gives identifications
$H_1(T, \bZ) \cong \bZ^2$
and
$H^1(T, \bZ) \cong \bZ^2$.
We equip $T$ with the orientation
coming from the standard orientation on $\bR^2$.

A {\em dimer model} is a finite bicolored graph $G = (B, W, E)$
embedded in $T$
such that
\begin{itemize}
 \item
$G$ has no univalent node, and
 \item
any connected component
of the complement
$T \setminus \bigcup_{e \in E} e$
of the graph is simply-connected.
\end{itemize}


\subsection{Perfect matchings and characteristic polygons}

A {\em perfect matching}
(or a {\em dimer configuration})
on a graph $(N, E, \partial)$
is a subset $D$ of $E$
such that for any node $n \in N$,
there is a unique edge $e \in D$
incident to $n$.
A dimer model is said to be {\em non-degenerate}
if for any edge $e \in E$,
there is a perfect matching $D$
such that $e \in D$.

Let $G = (B, W, E)$ be a dimer model, and
consider the bicolored graph $\Gtilde$ on $\bR^2$
obtained from $G$ by pulling-back
to the universal cover $\bR^2 \to T$.
The set of perfect matchings on $G$ is naturally identified
with the set of periodic perfect matchings
on the infinite graph $\Gtilde$ on the universal cover.
Fix a perfect matching $D_0$
called the {\em reference matching}.
For any
perfect matching $D$,
the union $D \cup D_0$
divides $\bR^2$ into connected components.
The {\em height function} $h_{D, D_0}$ is
a locally-constant function on
$\bR^2 \setminus (D \cup D_0)$
which increases (resp. decreases)
by $1$
when one crosses an edge $e \in D$
with the black (resp. white) node
on his right
or an edge $e \in D_0$
with the white (resp. black) node
on his right.
This rule determines the height function
up to an addition of a constant.
The height function may not be periodic
even if $D$ and $D_0$ are periodic,
and the {\em height change}
$h(D, D_0) = (h_x(D, D_0), h_y(D, D_0)) \in \bZ^2$
of $D$ with respect to $D_0$
is defined as the difference
\begin{align*}
 h_x(D, D_0) &= h_{D, D_0}(p+(1,0)) - h_{D, D_0}(p), \\
 h_y(D, D_0) &= h_{D, D_0}(p+(0,1)) - h_{D, D_0}(p) 
\end{align*}
of the height function,
which does not depend on the choice of
$p \in \bR^2 \setminus (D \cup D_0)$.
More invariantly,
height changes can be considered
as an element of $H^1(T, \bZ)$.
The dependence of the height change
on the choice of the reference matching
is given by
$$
 h(D, D_1) = h(D, D_0) - h(D_1, D_0)
$$
for any three perfect matchings $D$, $D_0$ and $D_1$.
We often suppress the dependence of the height difference
on the reference matching
and just write $h(D) = h(D, D_0)$.

For a fixed reference matching $D_0$,
the characteristic polynomial of $G$ is defined by
$$
 Z(x, y)
  = \sum_{D \in \perfmat(G)}
     x^{h_x(D)} y^{h_y(D)},
$$
where $\perfmat(G)$ is the set of perfect matchings on $G$.
The characteristic polynomial is
a Laurent polynomial in two variables,
whose Newton polygon gives the {\em characteristic polygon},
defined as the convex hull
$$
 \Delta = \Conv
 \{ (h_x(D), h_y(D)) \in \bZ^2
      \mid \text{$D$ is a perfect matching on $G$} \}
$$
of the set of height changes of perfect matchings
on the dimer model.

A {\em corner} of $\Delta$ is an extremal point of $\Delta$,
and a {\em side} of $\Delta$ is the interval
between two neighboring corners.
A side is divided into 
{\em primitive side segments},
defined as intervals between two adjacent lattice points
on the boundary of $\Delta$.
A perfect matching $D$ is said to be
a {\em corner perfect matching}
if its height change $h(D)$ is
on the corner of the characteristic polygon.
The {\em multiplicity} of a perfect matching $D$
is the number of perfect matchings
whose height changes are the same as $D$.

\subsection{Zigzag paths and their slopes}
 \label{sc:slope}

A {\em zigzag path} is a path on a bicolored graph
in an oriented surface
which makes a maximum turn to the right on a white node
and a maximal turn to the left on a black node.
We assume that a zigzag path does not have an endpoint,
so that it is either periodic or infinite in both directions.
Here, the latter can happen
only if the graph is infinite.
Figure \ref{fg:zigzag} shows an example
of a part of a dimer model
and a zigzag path on it.

\begin{figure}[htbp]
\centering
\begin{minipage}{.4 \linewidth}
\centering
\input{zigzag.pst}
\caption{A zigzag path}
\label{fg:zigzag} 
\end{minipage}
\begin{minipage}{.4 \linewidth}
\centering
\input{zigzag2.pst}
\caption{A path on the quiver along a zigzag path}
\label{fg:zigzag2} 
\end{minipage}
\end{figure}

Let $z$ be a zigzag path on a dimer model, and
assume that there is a perfect matching $D_0$
which intersect half of the edges
constituting $z$
(i.e., every other edge of $z$ belongs to $D_0$).
Then 
the height change of any other perfect matching $D$
with respect to $D_0$
in the direction of $z$ is negative;
\begin{equation} \label{eq:zigzag}
 \langle h(D, D_0), [z] \rangle \le 0.
\end{equation}
Here, $[z] \in H_1(T, \bZ) \cong \bZ^2$
is the homology class of $[z]$,
which is paired with the height change
considered as an element of $H^1(T, \bZ)$.
To show this, replace $z$ by the path $p$ on the quiver
going along $z$ (on the left side of $z$),
which belongs to the class $[z]$
as shown in \pref{fg:zigzag2}.
Then \eqref{eq:zigzag} follows from the fact that
as one goes around $T$ along $p$,
one crosses no edge in $D_0$
and every edge one crosses has a white node
on one's right.
In this way,
such a zigzag path gives an inequality
which bound the Newton polygon
of the characteristic polynomial.

The homology class $[z] = (u, v) \in H_1(T, \bZ) \cong \bZ^2$
of a zigzag path $z$
considered as an element of $\bZ^2$
will be called its {\em slope}.
If a zigzag path does not have a self-intersection,
then $(u, v) \in\bZ^2$ is a primitive element,
and we sometimes think of the slope
as an element
$$
 \frac{(u, v)}{\sqrt{u^2+v^2}} \in S^1
$$
of the unit circle.
The set of slopes has the natural
counter-clockwise cyclic order
as a subset of the unit circle.

\subsection{Quivers}

A {\em quiver} is an oriented graph,
which is a quadruple $(V, A, s, t)$ consisting of
\begin{itemize}
 \item a set $V$ of vertices,
 \item a set $A$ of arrows, and
 \item two maps $s, t: A \to V$ from $A$ to $V$.
\end{itemize}
For an arrow $a \in A$,
the vertices $s(a)$ and $t(a)$
are called the {\em source}
and the {\em target} of $a$
respectively.

A {\em path} on a quiver
is an ordered set of arrows
$(a_n, a_{n-1}, \dots, a_{1})$
such that $s(a_{i+1}) = t(a_i)$
for $i=1, \dots, n-1$.
We also allow for a path of length zero,
starting and ending at the same vertex.

The {\em path algebra} $\bC Q$
of a quiver $Q = (V, A, s, t)$
is the algebra
spanned by the set of paths
as a vector space,
and the multiplication is defined
by the concatenation of paths;
$$
 (b_m, \dots, b_1) \cdot (a_n, \dots, a_1)
  = \begin{cases}
     (b_m, \dots, b_1, a_n, \dots, a_1) & s(b_1) = t(a_n), \\
      0 & \text{otherwise}.
    \end{cases}
$$

A {\em quiver with relations}
is a pair of a quiver
and a two-sided ideal $\scI$
of its path algebra.
For a quiver $\Gamma = (Q, \scI)$
with relations,
its path algebra $\bC \Gamma$ is defined as
the quotient algebra $\bC Q / \scI$.

\subsection{A quiver with relations
associated with a dimer model}
 \label{sc:dimer_quiver}

A dimer model $(B, W, E)$ encodes
the information of a quiver
$\Gamma = (V, A, s, t, \scI)$
with relations
in the following way:
The set $V$ of vertices
is the set of connected components
of the complement
$
 T \setminus (\bigcup_{e \in E} e),
$
and
the set $A$ of arrows
is the set $E$ of edges of the graph.
The orientations of the arrows are determined
by the colors of the nodes of the graph,
so that the white node $w \in W$ is on the right
of the arrow.
In other words,
the quiver is the dual graph of the dimer model
equipped with an orientation given by
rotating the white-to-black flow on the edges of the dimer model
by minus 90 degrees.

The relations of the quiver are described as follows:
For an arrow $a \in A$,
there exist two paths $p_+(a)$
and $p_-(a)$
from $t(a)$ to $s(a)$,
the former going around the white node
incident to $a \in E = A$ clockwise,
and the latter going around the black node
incident to $a$ counterclockwise
as shown in \pref{fg:relation}.
Then the ideal $\scI$
of the path algebra is
generated by $p_+(a) - p_-(a)$
for all $a \in A$.

\begin{figure}[ht]
\centering
\begin{minipage}{.4 \linewidth}
\centering
\input{relation.pst}
\caption{Relations on the quiver}
\label{fg:relation}
\end{minipage}
\begin{minipage}{.4 \linewidth}
\centering
\input{small_cycle.pst}
\caption{Small cycles}
\label{fg:small_cycle}
\end{minipage}
\end{figure}

\subsection{Small cycles, minimal paths and
weak equivalence}
\label{sc:small_minimal_weak}

A {\em small cycle} on a quiver
associated with a dimer model
is a path obtained as
the product of arrows surrounding a node of the dimer model.
Three small cycles are shown in \pref{fg:small_cycle}.
A path $p$ is said to be {\em minimal}
if it is not equivalent to a path containing a small cycle.

Note that small cycles starting from a fixed vertex
are equivalent to each other.
It follows that the sum $\omega := \sum_{v \in V} \omega_v$ of small cycles over the set of vertices,
where one picks one small cycle $\omega_v$ for each vertex $v$,
is a well-defined element of the path algebra
independent of the choice of $\omega_v$.
One can easily see that the element $\omega$ belongs to the center of the path algebra,
and there is the universal map
$$
 \bC \Gamma \to \bC \Gamma[\omega^{-1}]
$$
into the localization of the path algebra
by the multiplicative subset generated by $\omega$.
Two paths are called {\em weakly equivalent}
if they give the same element
in $\bC \Gamma[\omega^{-1}]$.

Suppose that there is a perfect matching $D$.
Note that every small cycle contains exactly one arrow in $D$.
Then \cite[Lemma 2.1]{Ishii-Ueda_CCDM} implies that
two paths with the same source and the target
are weakly equivalent if and only if they
have the same homology class and they
contain the same number of arrows in $D$.

\subsection{Moduli space of quiver representations}

A {\em representation} of a quiver
$
\Gamma = (V, A, s, t, \scI)
$
with relations is a module
over the path algebra $\bC  \Gamma$.
In other words,
a representation of $\Gamma$ is a collection
$((V_v)_{v \in V}, (\psi_a)_{a \in A})$
of vector spaces $V_v$ for $v \in V$
and linear maps $\psi_a : V_{s(a)} \to V_{t(a)}$
for $a \in A$ satisfying relations in $\scI$.
The {\em dimension vector} of a representation
$((V_v)_{v \in V}, (\psi_a)_{a \in A})$
is given by $(\dim V_v)_{v \in V} \in \bZ^V$.
This allows us to think of $\bZ^V$
as a quotient of the Grothendieck group
of the abelian category
of finite dimensional representations of $\Gamma$.
The {\em support} of a representation
is the set of vertices $v \in V$
such that $\dim V_v \ne 0$.

A {\em stability parameter} $\theta$ is
an element of $\Hom (\bZ^V, \bZ)$.
A $\bC \Gamma$-module $M$
is said to be {\em $\theta$-stable}
if $\theta(M)=0$ and for any non-trivial submodule
$N \subsetneq M$,
one has $\theta(N) > \theta(M)$.
$M$ is {\em $\theta$-semistable}
if $\theta(N) \ge \theta(M)$ holds
instead of $\theta(N) > \theta(M)$.
A stability parameter $\theta$ is said to be {\it generic}
with respect to a fixed dimension vector
if semistability implies stability.
This stability condition is introduced by King \cite{King}
to construct the moduli space $\scM_{\theta}$
representing (the sheafification of) the functor
$$
\begin{array}{ccc}
 (\scS ch) & \to & (\scS et) \\
  \vin & & \vin \\
  T & \mapsto & (\text{a flat family over $T$
  of $\theta$-stable representations of $\Gamma$} )/ \sim
\end{array}
$$
for a fixed dimension vector.
Here, a {\em flat family of representations} of $\Gamma$
over $T$
is a collection $(\scL_v)_{v \in V}$
of vector bundles on $T$
for each vertex $v$ of $\Gamma$ and
a collection $(\phi_a)_{a \in A}$ of morphisms
$\phi_a : \scL_{s(a)} \to \scL_{t(a)}$
for each arrow $a$ of $\Gamma$
satisfying the relations $\scI$ of $\Gamma$.
Two families are defined to be equivalent
if they are isomorphic up to
tensor product $\scL_v \mapsto \scL_v \otimes \scL$
by some line bundle $\scL$
simultaneously for all vertices $v \in V$.
If the dimension vector is a primitive vector,
then we do not have to sheafify the functor,
and there is a universal family over the moduli space.
The bundles $\scL_v$ in the universal family are called
the {\it tautological bundles}.
In the rest of this paper, $\scM_\theta$ denotes
the moduli space of $\theta$-stable $\bC\Gamma$-modules
for the dimension vector $(1,1, \dots,1)$.
On the other hand,
the moduli space $\scMbar_\theta$
of $\theta$-semistable modules
does not represent the moduli functor,
but parametrizes S-equivalence classes of
$\theta$-semistable modules.

\subsection{Perfect matchings and moduli spaces}
 \label{sc:pm_moduli}

The main theorem of \cite{Ishii-Ueda_08} states that
when a dimer model is non-degenerate,
then the moduli space $\scM_{\theta}$ is a smooth Calabi-Yau
toric 3-fold for generic $\theta$.
A description of the universal representation
around each torus fixed points in terms of local coordinates
is given in \cite[Lemma 4.5]{Ishii-Ueda_08},
which immediately implies the following:
\begin{lemma}\label{lm:reduced}
Let $G$ be a non-degenerate dimer model.
Then for each arrow $a$ of the associated quiver,
the zero locus of $\phi_a : \scL_{s(a)} \to \scL_{t(a)}$
is a reduced subscheme of $\scM_{\theta}$.
Moreover, for each vertex $v$, the zero locus of the
map $\scL_v \to \scL_v$ corresponding to the small cycle
is the union of all the toric divisors with multiplicities one.
\end{lemma}
It is also proved in \cite[Section 6]{Ishii-Ueda_08} that
a toric divisor in $\scM_{\theta}$ gives a perfect matching
in such a way that the stabilizer group of the divisor
is given by the height change of the perfect matching.

A perfect matching can be considered as a set of walls
which block some of the arrows;
for a perfect matching $D$,
let $Q_D$ be the subquiver of $Q$
whose set of vertices is the same as $Q$
and whose set of arrows consists of $A \setminus D$
(recall that $A = E$).
The path algebra $\bC Q_D$ of $Q_D$ is a subalgebra of $\bC Q$,
and the ideal $\scI$ of $\bC Q$
defines an ideal $\scI_D = \scI \cap \bC Q_D$ of $\bC Q_D$.
A path $p \in \bC Q$ is said to be an {\em allowed path}
with respect to $D$ if $p \in \bC Q_D$.

With a perfect matching,
one can associate a representation of the quiver
with dimension vector $(1, \dots, 1)$
by sending any allowed path to $1$ and
other paths to $0$.
A perfect matching is said to be {\em simple}
if this representation is simple,
i.e., has no non-trivial subrepresentation.
This is equivalent to the condition
that there is an allowed path
starting and ending at any given pair of vertices.

\subsection{Quivers as categories}\label{sc:quiver-category}

With a quiver $\Gamma$ with relations,
one can associate a $\bC$-linear category $\scC$
in the following way:
\begin{itemize}
\item
The set of objects of $\scC$ is the set of vertices of $\Gamma$.
\item
The space of morphisms between two objects $v$ and $w$
is the vector space
$e_w \cdot \bC \Gamma \cdot e_v$
where $e_v$ and $e_w$ are idempotents of the path algebra
corresponding to the vertices $v$ and $w$ of $\Gamma$.
\item
The composition of morphisms comes from
the product in the path algebra.
\end{itemize}
In terms of the category $\scC$,
a representation of $\Gamma$ is just a linear functor
from $\scC$ to the category of vector spaces.

The advantage of working with categories
rather than path algebras is the following:
Let $v$ and $w$ be two vertices in a quiver
$
 \Gamma = (V, A, s, t, \scI)
$
with relations and
$\{ a_1, \ldots, a_r \}$ be any subset
of the set of arrows of $\Gamma$
from $v$ to $w$.
Then we can define another quiver
$
 \Gamma' = (V', A', s', t', \scI')
$
by setting
$
 V' = V \setminus \{ v \},
$
$
 A' = A \setminus \{ a_1, \ldots, a_r \},
$
and
\begin{align*}
 s'(a) &=
 \begin{cases}
  s(a) & s(a) \ne v, \\
  w & s(a) = v,
 \end{cases} &
 t'(a) &=
 \begin{cases}
  t(a) & t(a) \ne v, \\
  w & t(a) = v.
 \end{cases}
\end{align*}
The relations of $\Gamma'$ is determined
by the condition that $\bC \Gamma'$ is Morita equivalent
to the localization of $\bC \Gamma$
at the arrows $a_1, \ldots, a_r$.
This means that $\Gamma'$ is obtained from $\Gamma$
by inverting the arrows $a_1, \ldots, a_r$ and
identifying two vertices $v$ and $w$
which become isomorphic
after the inversion of the arrows.
There is a natural map $\pi : \bC \Gamma \to \bC \Gamma'$
between path algebras,
which is {\em not} an algebra homomorphism
since
$$
\pi(e_w) \circ \pi(e_v)
 = e_w \circ e_w
 = e_w
 \ne 0
 = \pi(0)
 = \pi(e_w \circ e_v).
$$
Nevertheless,
the map $\pi$ induces
a functor $\varpi : \scC \to \scC'$ from the category $\scC$
associated with $\Gamma$
to the category $\scC'$
associated with $\Gamma'$.
Since a representation of $\Gamma$
is a functor from $\scC$
to the category of vector spaces,
the functor $\varpi$ induces a functor
$
 \varpi^* : \module \bC \Gamma' \to \module \bC \Gamma
$
between categories of representations.
The image of the functor $\varpi^*$ consists of
representations $((V_v)_{v \in V}, (\psi_a)_{a \in A})$
such that $V_v = V_w$ and
$\psi_{a_1} = \cdots = \psi_{a_r} = \id_{V_v}$.

\subsection{Example}

\begin{figure}
\centering
\begin{minipage}{.4 \linewidth}
\centering
\input{dP1_graph.pst}
\caption{A dimer model}
\label{fg:dP1_graph}
\end{minipage}
\begin{minipage}{.45 \linewidth}
\centering
\input{dP1_quiver.pst}
\caption{The corresponding quiver}
\label{fg:dP1_quiver}
\end{minipage}
\end{figure}

\begin{figure}
\begin{minipage}{\linewidth}
\centering
\input{dP1_matchings.pst}
\caption{Eight perfect matchings}
\label{fg:dP1_matchings}
\end{minipage}
\end{figure}

\begin{figure}
\begin{minipage}{\linewidth}
\centering
\input{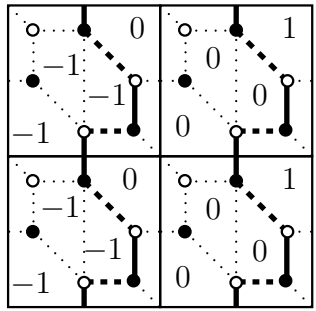}
\caption{The height function $h_{D_1,D_5}$}
\label{fg:dP1_height}
\end{minipage}
\end{figure}

\begin{figure}
\begin{minipage}{\linewidth}
\centering
\input{dP1_zigzags.pst}
\caption{Four zigzag paths}
\label{fg:dP1_zigzags}
\end{minipage}
\end{figure}

\begin{figure}
\begin{minipage}{\linewidth}
\centering
\input{dP1_diagram.pst}
\caption{The characteristic polygon}
\label{fg:dP1_diagram}
\end{minipage}
\end{figure}

As an example, consider the dimer model
in \pref{fg:dP1_graph}.
The corresponding quiver is shown in \pref{fg:dP1_quiver}.
This dimer model is non-degenerate,
and has eight perfect matchings $D_1, \dots, D_8$
shown in \pref{fg:dP1_matchings}.
The height function $h_{D_1, D_5}$
of $D_1$ with respect to $D_5$ is shown in \pref{fg:dP1_height}.
The characteristic polynomial is given by
$$
 Z(x, y) = 4 + x + y + \frac{1}{x} + \frac{1}{x y}.
$$
This dimer model has four zigzag paths
as shown in \pref{fg:dP1_zigzags}.
Note that the homology class of these four paths
are normal to the sides of the characteristic polygon
as shown in \pref{fg:dP1_diagram}.

\subsection{McKay quiver and hexagonal dimer models}
 \label{sc:McKay}

Let $\bTtilde \subset \GL(3, \bC)$ be the subgroup
consisting of diagonal matrices and put
$
 \bTtilde_0 = \bTtilde \cap \SL(3, \bC).
$
For a finite subgroup $A \subset \bTtilde_0$,
the character group
$
 A^* = \Hom(A, \bCx)
$
is a quotient of $\bTtilde_0^* \cong \bZ^2$,
and hence a quotient of $\bTtilde^* \cong \bZ^3$.
Let $\rho_x, \rho_y, \rho_z \in A^*$ be the images
of the coordinate functions
$x, y, z \in \bTtilde^*$ respectively.
The {\it McKay quiver} for $A$ has
$A^*$ as the set of vertices, and
there are three arrows
starting from each vertex $\rho$,
whose targets are
$\rho \rho_x$, $\rho \rho_y$ and $\rho \rho_z$ respectively.
We say that these arrows correspond to
``multiplications by $x$, $y$, $z$" respectively.
If $M_0$ denotes the kernel of the surjection $\bTtilde_0^* \to A^*$,
then the McKay quiver can be embedded
in the torus $T=(\bTtilde_0^* \otimes \bR)/M_0$,
and comes from a hexagonal dimer model on $T$
as in \cite{Reid_MC}
(see also \cite[Section 5]{Ueda-Yamazaki_NBTMQ}
and an example in Section \ref{ss:algorithm_examples} below).
The corresponding path algebra with relations is isomorphic
to the crossed product algebra $\bC[x, y, z] \rtimes A$.
The Hilbert scheme $\ahilb(\bC^3)$ of $A$-orbits, parameterizing {\it $A$-clusters}, is isomorphic to the moduli space $\scM_\theta$ for this quiver
with resect to a stability parameter $\theta$ such that $\theta(\rho) > 0$ for every non-trivial $\rho \in A^*$
(cf. e.g. \cite[Section 3]{Ito-Nakajima}).

\section{Consistency conditions on dimer models}
 \label{sc:consistency}

\subsection{Divalent node}

Let $G = (B, W, E)$ be a non-degenerate dimer model.
For a divalent node $n \in B \sqcup W$,
one can contract two nodes adjacent to $n$
and obtain another dimer model $G' = (B', W', E')$
as shown in Figure \ref{fg:remove_divalent_node}.
Note that the two nodes adjacent to $n$ must be distinct
since the dimer model is non-degenerate.
The numbers of black nodes and white nodes
are reduced by one,
and the number of edges is reduced by two
under this operation.
If $G'$ still has a divalent node,
then one can continue this process
until the dimer model contains no divalent nodes.
It is clear from the definition of the zigzag paths
that there is a natural bijection
between the sets of zigzag paths on dimer models
before and after the removal of divalent nodes.
It is also clear from the definition
of the relations of the quiver
associated with a dimer model that
the isomorphism class of the path algebra
does not change
under the operation of removing divalent nodes.

Although divalent nodes do not cause any problem
for the purpose of this paper,
it is often convenient to assume
that all the divalent nodes are removed
to simplify the exposition.

\begin{figure}[htbp]
\centering
\input{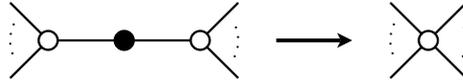}
\caption{Removal of a divalent node}
\label{fg:remove_divalent_node}
\end{figure}

\subsection{Consistent dimer models}

The following definition
is taken from \cite[Definition 3.5]{Ishii-Ueda_CCDM}.
It originates from the work of Hanany and Vegh
\cite{Hanany-Vegh},
and also studied by Bocklandt
\cite{Bocklandt_CCDM}.

\begin{definition} \label{df:consistency}
A dimer model is {\em consistent} if
\begin{itemize}
 \item
there is no homologically trivial zigzag path,
 \item
no zigzag path has a self-intersection
on the universal cover, and
 \item
no pair of zigzag paths on the universal cover intersect each other
in the same direction more than once.
\end{itemize}
\end{definition}

Here, two zigzag paths on a dimer model
are said to {\em intersect}
if they share an edge
(not a node)
after removing all the divalent node
from the dimer model.
One intersection consists of an odd number
of consecutive edges
connected by divalent nodes,
which must be just one edge
if the dimer model has no divalent node.
See Figure \ref{fg:zigzag-intersection}
for examples of an intersection
and a non-intersection.

\begin{figure}[htbp]
\centering
\input{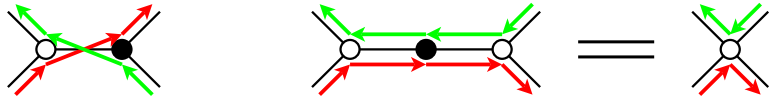}
\caption{Examples of an intersection (left) and
a non-intersection (right)}
\label{fg:zigzag-intersection}
\end{figure}

The third condition means that
if a pair $(z, w)$ of zigzag paths on the universal cover
has two intersections $a$ and $b$
and the zigzag path $z$ points from $a$ to $b$,
then the other zigzag path $w$ must point from $b$ to $a$. 

Figure \ref{fg:inconsistent-2-zigzag} shows an example
of a part of an inconsistent dimer model
which contains a homologically trivial zigzag path.

\begin{figure}[htbp]
\centering
\input{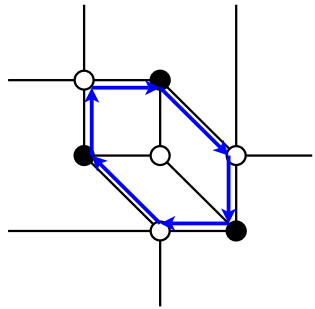}
\caption{A homologically trivial zigzag path}
\label{fg:inconsistent-2-zigzag}
\end{figure}
Figure \ref{fg:inconsistent} shows an inconsistent dimer model,
which contains a pair of zigzag paths
on the universal cover
intersecting in the same direction twice
as in Figure \ref{fg:inconsistent_zigzag}.
\begin{figure}[htbp]
\centering
\input{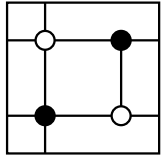}
\caption{An inconsistent dimer model}
\label{fg:inconsistent}
\end{figure}

\begin{figure}[htbp]
\centering
\input{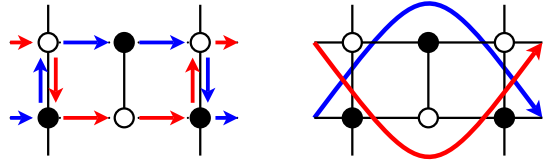}
\caption{A pair of zigzag paths in the same direction
intersecting twice}
\label{fg:inconsistent_zigzag}
\end{figure}

On the other hand,
a pair of zigzag paths going in the opposite direction
may intersect twice in a consistent dimer model.
Figure \ref{fg:P1P1_II_zigzag} shows a pair of such zigzag paths
on a consistent dimer model
in Figure \ref{fg:P1P1_II}.

\begin{figure}[htbp]
\centering
\input{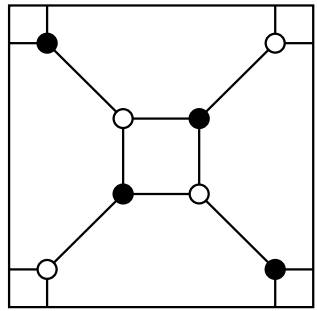}
\caption{A consistent non-isoradial dimer model}
\label{fg:P1P1_II}
\end{figure}

\begin{figure}[htbp]
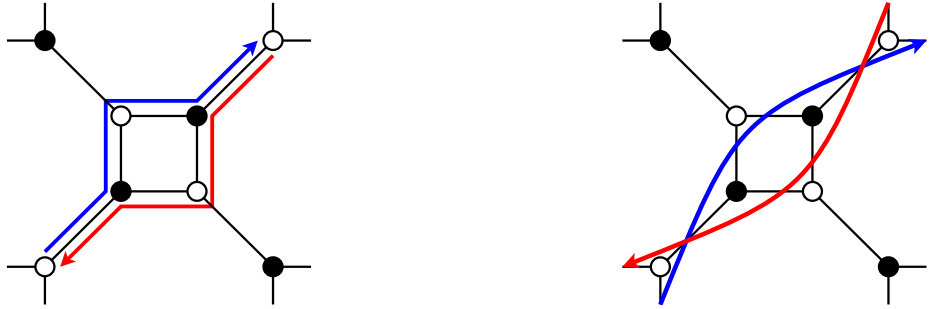

\begin{minipage}{.5 \linewidth}
\centering
\input{consistent_1.pst}
\end{minipage}
\begin{minipage}{.5 \linewidth}
\centering
\input{consistent_2.pst}
\end{minipage}
\caption{A pair of zigzag paths in the opposite direction
intersecting twice}
\label{fg:P1P1_II_zigzag}
\end{figure}

\subsection{Isoradial dimer models}

The following notion is due to
Duffin \cite{Duffin} and Mercat \cite{Mercat_DRSIM}:

\begin{definition} \label{df:isoradial}
A dimer model is {\em isoradial}
if one can choose an embedding of the graph into the torus
so that every face of the graph is a polygon
inscribed in a circle of a fixed radius
with respect to a flat metric on the torus.
Here, the circumcenter of any face must be contained
in the face.
\end{definition}

A dimer model is isoradial
if and only if zigzag paths behave like straight lines:

\begin{theorem}[%
{Kenyon and Schlenker \cite[Theorem 5.1]{Kenyon-Schlenker}}]
 \label{th:Kenyon-Schlenker}
A dimer model is isoradial
if and only if the following conditions are satisfied:
\begin{enumerate}
 \item
Every zigzag path is a simple closed curve.
 \item
Any pair of zigzag paths on the universal cover
share at most one edge.
\end{enumerate}
\end{theorem}

By comparing Theorem \ref{th:Kenyon-Schlenker}
with Definition \ref{df:consistency},
one obtains the following:

\begin{corollary} \label{cr:isoradial_consistent}
Isoradial dimer models are consistent.
\end{corollary}

The converse to Corollary \ref{cr:isoradial_consistent}
does not hold:
The dimer model shown in Figure \ref{fg:P1P1_II}
gives an example of a consistent dimer model
which is not isoradial.
A trivial example of a consistent dimer model
which is not isoradial
can be obtained by adding a divalent node
to any isoradial dimer model.

\subsection{Properly-ordered dimer models}

For a node in a dimer model,
the set of zigzag paths going through the edges adjacent to it
has a natural cyclic ordering
given by the directions of the outgoing paths
from the node.
On the other hand,
the homology classes of these zigzag paths determine
another cyclic ordering
if these classes are distinct.
The following condition is introduced by Gulotta:

\begin{definition}[{Gulotta \cite[Section 3.1]{Gulotta}%
}]
 \label{df:properly-ordered}
A dimer model is {\em properly ordered} if
\begin{itemize}
\item
there is no homologically trivial zigzag path,
\item
no zigzag path has a self-intersection on the universal cover,
 \item
no pair of zigzag paths in the same homology class have a common node, and
 \item
for any node of the dimer model,
the natural cyclic order on the set of zigzag paths
going through that node
coincides with the cyclic order
determined by their homology classes.
\end{itemize}  
\end{definition}

This condition is equivalent
to the consistency condition
in Definition \ref{df:consistency}:

\begin{proposition}[{\cite[Proposition 4.4]{Ishii-Ueda_CCDM}}]
 \label{prop:properly-ordered-consistency}
A dimer model is consistent
if and only if it is properly-ordered.
\end{proposition}

\subsection{The first consistency condition}

Mozgovoy and Reineke
\cite[Condition 4.12]{Mozgovoy-Reineke}
introduced the following condition:

\begin{definition} \label{df:MR}
A dimer model is said to satisfy
the {\em first consistency condition}
in the sense of Mozgovoy and Reineke
if weakly equivalent paths are equivalent.
\end{definition}

The consistency condition
in Definition \ref{df:consistency}
implies this condition:

\begin{lemma}[{\cite[Lemma 3.10]{Ishii-Ueda_CCDM}}]
A consistent dimer model satisfies
the first consistency condition
in the sense of Mozgovoy and Reineke.
\end{lemma}

Mozgovoy and Reineke \cite{Mozgovoy-Reineke}
proved that
the path algebra of the quiver with relation
coming from a dimer model
is a Calabi-Yau 3 algebra
in the sense of Ginzburg \cite{Ginzburg_CYA}
if the dimer model satisfies
the first consistency condition
and one extra condition
which they call the second consistency condition.
The latter condition is shown to be redundant
by Davison \cite{Davison}.
Broomhead has proved the Calabi-Yau 3 property
of the path algebra
for isoradial dimer models
\cite{Broomhead}.
The proof of Theorem \ref{th:main} in this paper
does not rely on any of these results, and
gives an independent proof
of the Calabi-Yau 3 property
of the path algebra
of the quiver with relations
associated with a consistent dimer model
through the derived equivalence
$
 D^b \coh \scM_\theta \cong D^b \module \bC \Gamma.
$

\section{Adjacent zigzag paths and large hexagons}
 \label{sc:large-hexagon}

In this section,
we assume for simplicity
that all divalent nodes are removed from the dimer model.
In this case,
a pair of zigzag paths intersect each other
if and only if they share a common edge,
and one intersection consists of exactly one edge.

\subsection{Adjacent zigzag paths}

Recall from Section \ref{sc:slope}
that the {\em slope} of a zigzag path
on a dimer model is its homology class
considered as an element in $\bZ^2$.
The lack of self-intersection of a zigzag path
in a consistent dimer model
implies the primitivity of its slope.
There may be several zigzag paths
with a given slope.
The set of slopes naturally has a cyclic order,
and a pair of zigzag paths are said
to have {\em adjacent slopes}
if their slopes are adjacent
with respect to this cyclic order.

The following three lemmas are immediate consequences
of Proposition \ref{prop:properly-ordered-consistency}:

\begin{lemma}
 \label{lm:zigzag-adjacency}
If a pair of zigzag paths
in a consistent dimer model
intersect each other more than once on the universal cover,
then their slopes are not adjacent.
\end{lemma}

\begin{proof}
Assume that there is a pair $(a, b)$ of zigzag paths
intersecting twice in the opposite direction
as in Figure \ref{fg:zigzag-non-adjacency}.
\begin{figure}
\centering
\input{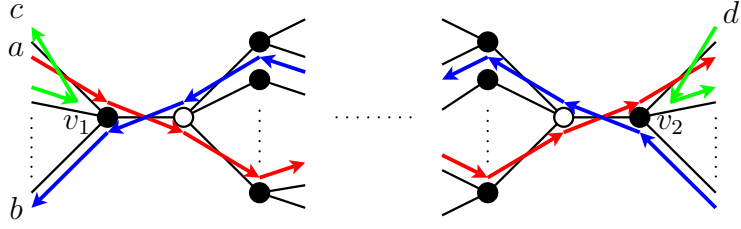}
\caption{A pair of zigzag paths intersecting twice}
\label{fg:zigzag-non-adjacency}
\end{figure}
Let $v_1$ and $v_2$ be the vertices
adjacent to the first and the last edges
where $a$ and $b$ intersect.
Then there are two other zigzag paths $c$ and $d$
such that $c$ intersects with $a$
at the edge adjacent to the vertex $v_1$
and $d$ intersects with $a$
at the edge adjacent to the vertex $v_2$.
Then the slopes of $c$ and $d$ must come
in between $a$ and $b$ by Proposition \ref{prop:properly-ordered-consistency},
preventing them to be adjacent.
\end{proof}

\begin{lemma}
 \label{lm:zigzag-adjacency2}
If a pair of zigzag paths
in a consistent dimer model
have common node other than their intersection.
Then the slopes of this pair of zigzag paths
are not adjacent.
\end{lemma}

\begin{proof}
Since the dimer model is consistent,
it is properly-ordered
by Proposition \ref{prop:properly-ordered-consistency}.
If a pair of zigzag paths have a common node
other than their intersection,
then they are not adjacent
with respect to the cyclic order
around that node.
Now it follows
from Definition \ref{df:properly-ordered}
that their slopes are not adjacent.
\end{proof}

\begin{lemma} \label{lemma:independent}
If a dimer model is consistent,
then there is a pair of zigzag paths
with linearly independent slopes.
\end{lemma}

\begin{proof}
A dimer model always have a node
with valence greater than two.
Then there are at least three zigzag paths at the node
whose slopes are different by the properly-orderedness.
\end{proof}

\subsection{Large hexagons}
Lemmas \ref{lm:zigzag-adjacency}
and \ref{lm:zigzag-adjacency2}
show that a pair of zigzag paths
with adjacent slopes
in a consistent dimer model
behave like a pair of lines;
they have no self-intersection,
and any pair of lifts to the universal cover
intersect exactly once.
Any pair of lines on a torus
divides the torus into parallelograms.
Since an intersection of a pair of zigzag paths
in a consistent dimer model consists
of an edge instead of a point,
they divide the torus into hexagons
instead of parallelograms.

\begin{definition}
Let $G = (B, W, E)$ be a consistent dimer model
on a torus $T$
and $(z, w)$ be a pair of zigzag paths on $G$
with adjacent slopes.
A {\em large hexagon}
is a connected component of
the complement $T \setminus (z \cup w)$
of the union of the pair of zigzag paths.
\end{definition}

Figure \ref{fg:large_square_tile} shows a part of
a large square tiling,
and an example of a collection of zigzag paths
with adjacent slopes is shown
in Figure \ref{fg:large_square_tile_zigzag}.
One can see that these zigzag path
divides the torus into large hexagons 
as shown in Figure \ref{fg:large_square_tile_zigzag_hexagon}.

\begin{figure}[htbp]
\centering
\input{large_square_tile.pst}
\caption{A part of a large square tiling}
\label{fg:large_square_tile}
\end{figure}

\begin{figure}[htbp]
\centering
\input{large_square_tile_zigzag.pst}
\caption{A pair of zigzag paths with adjacent slopes}
\label{fg:large_square_tile_zigzag}
\end{figure}

\begin{figure}[htbp]
\centering
\input{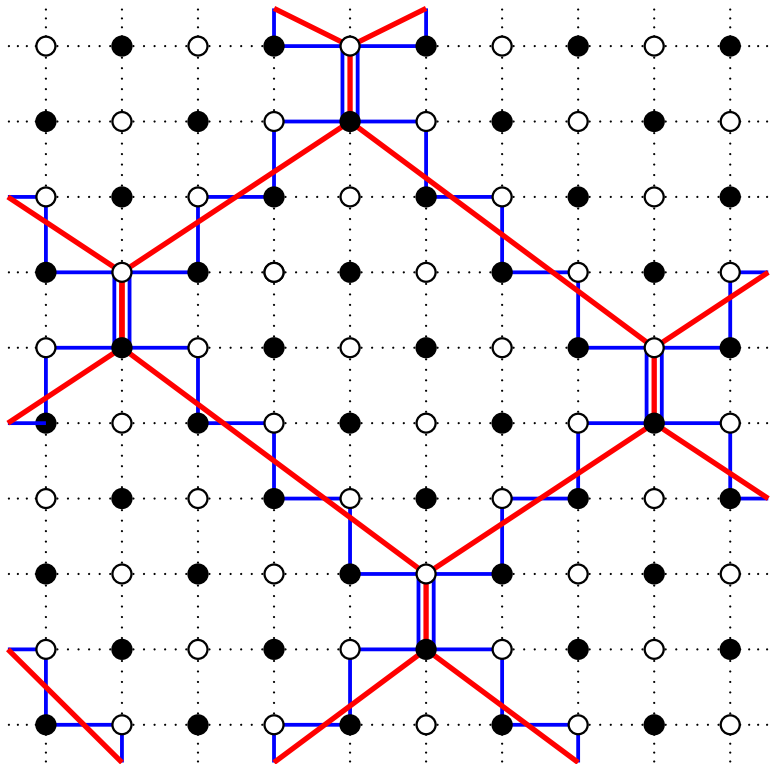}
\caption{Large hexagons}
\label{fg:large_square_tile_zigzag_hexagon}
\end{figure}

\begin{figure}[htbp]
\centering
\input{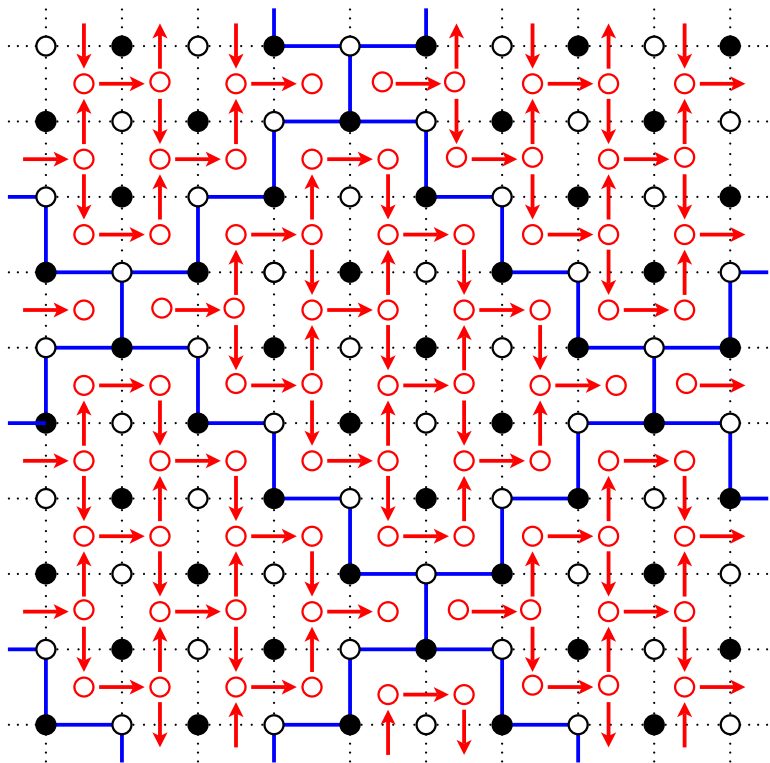}
\caption{Sources and sinks}
\label{fg:large_square_tile_zigzag_quiver}
\end{figure}

By removing arrows dual to edges
in the pair of zigzag paths,
the quiver associated with the dimer model
is divided into disjoint union of subquivers,
each of whose connected components
are in one-to-one correspondence
with a large hexagon.
Inside each such connected subquiver,
there are a pair of distinguished vertices
called the {\em source} and the {\em sink}.
The source vertex is characterized by the existence
of a path from that vertex
to any other vertex in the subquiver,
and the sink vertex is characterized by the dual property
that there is a path of the subquiver
from any other vertex
to the sink vertex.
The arrow dual to an intersection
of the pair of zigzag paths goes
from the source vertex of one large hexagon
to the sink vertex of an adjacent large hexagon.
See Figure \ref{fg:large_square_tile_zigzag_quiver}
for an example of the subquivers
and their source and sink vertices.

\subsection{Large hexagons and the McKay quiver}
 \label{sc:large-McKay}

The tessellation by large hexagons forms a new dimer model,
and as in Section \ref{sc:McKay},
the resulting quiver $\Lambda$ with relations can be identified
with the McKay quiver for a suitable finite subgroup $A \subset \SL(3, \bC)$
acting on $\bC^3 = \Spec \bC[x, y, z]$
in the following way:
\begin{itemize}
 \item
Choose any vertex of $\Lambda$ and
identify it with the trivial representation.
 \item
The arrow dual to an intersection of
the two zigzag paths
is identified with ``multiplication by $z$''.
 \item
The cyclic order of three arrows
starting from a vertex of $\Lambda$
coming from the orientation of the torus is given by $(x, y, z)$.
\end{itemize}
The fact that we have taken only a pair of zigzag paths
with adjacent slopes,
so that there is only one zigzag path
in each slope,
implies that $\rho_x$ generates the character group $A^*$,
and so does $\rho_y$ under the notation in Section \ref{sc:McKay}.
Hence the subgroup $A \subset  \SL(3, \bC)$
is obtained by embedding a finite small subgroup $A \subset \GL(2, \bC)$
into $\SL(3, \bC)$.

\section{Consistent dimer models are non-degenerate}
 \label{sc:non-degeneracy}

We prove the following in this section:

\begin{proposition} \label{prop:non-degenerate}
A consistent dimer model is non-degenerate.
\end{proposition}

\begin{proof}
We may assume there is no divalent node.
For an edge $e$ in a consistent dimer model,
choose a zigzag path $z$ containing the edge.
Choose another zigzag path $w$
whose slope is adjacent to that of $z$.
Then $z$ and $w$ divide the torus into large hexagons.
In each large hexagon,
there are two paths $p$ and $q$
from the source to the sink
along $z \cup w$
as shown in Figure \ref{fg:large_square_tile_u}.
\begin{figure}[htbp]
\centering
\input{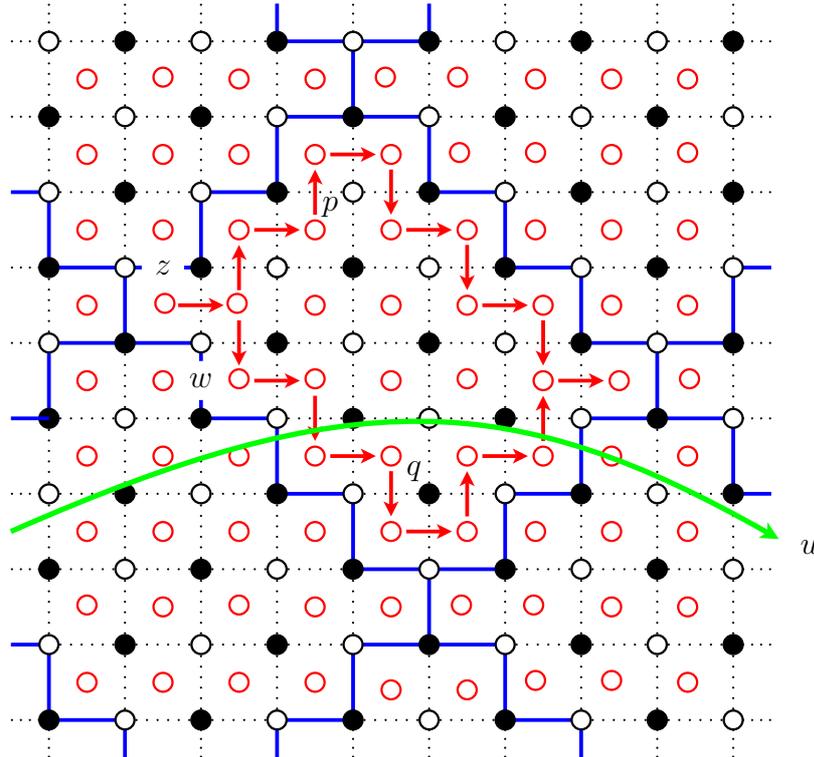}
\caption{Two minimal paths $p$ and $q$ inside a large hexagon}
\label{fg:large_square_tile_u}
\end{figure}
One path $p$ starts from the source vertex,
goes along the zigzag path $z$
until $z$ intersects with $w$,
from which point the path goes along $w$
and arrive at the sink vertex.
The other path $q$ starts from the source vertex,
goes along the zigzag path $w$
until $w$ intersects with $z$,
from which point the path goes along $z$.
Both $p$ and $q$ are minimal:
Assume that one of these paths are not minimal.
Then there is another zigzag path $u$
which intersect this path in the same direction
more than once
by \cite[Lemma 3.11]{Ishii-Ueda_CCDM}.
This implies that the slope of the zigzag path $u$ comes
in between the slopes of $z$ and $w$
with respect to the natural cyclic order
on the set of slopes
as shown in Figure \ref{fg:large_square_tile_u}.
This contradicts the adjacency of slopes
of $z$ and $w$.

Take the set $D_1$ of every other edges
on the union of $z$ and $w$
starting from the edge $e$,
and take the union $D = D_1 \cup D_2$
with the set $D_2$ of edges
in the interiors of the large hexagons
which are not crossed by any minimal path
from the source to the sink.
See Figure \ref{fg:large_square_tile_zigzag_pm}
for an example
when $e$ is at the intersection of $z$ and $w$.
We show that $D$ is a perfect matching:

Let $n$ be a node on the union
$z \cup w$ of the zigzag paths.
Then it is clear from the construction
that there is a unique edge in $D_1$ adjacent to $n$
and no edge in $D_2$ is adjacent to $n$.

Take a node $n$ in the interior of a large hexagon.
We show that there is a unique edge in $D_2$ connected to $n$.
Since $p_1$ and $p_2$ are minimal paths
with the same source and the target,
they are equivalent.
Since $p_1$ and $p_2$ are not homotopic
in $T \setminus \{n\}$,
there are two minimal paths $q_1$ and $q_2$
from the source vertex to the sink vertex
inside the large hexagon
such that
$p_1$ is homotopic to $q_1$ in $T \setminus \{n\}$,
$p_2$ is homotopic to $q_2$ in $T \setminus \{n\}$, and
$q_2$ is obtained from $q_1$
by replacing $p_+(a)$ by $p_-(a)$
for an arrow $a \in A=E$.
Then $a$ must be adjacent to $n$ and
either $q_1$ or $q_2$ passes through
all edges incident to $n$ except $a$.
Hence it suffices to show $a \in D_2$.
Let $r$ be a minimal path from the source to the sink.
Then $r$ intersects neither $z$ nor $w$
by \cite[Lemma 3.11]{Ishii-Ueda_CCDM}
and hence $r$ stays inside the large hexagon.
Take a zigzag path $y$ which passes through $a$.
Since the dimer model is consistent and $z$ and $w$ have adjacent slopes,
$y$ divides the large hexagon into two connected components
such that the source and the sink are not in the same component.
By \cite[Lemma 3.7]{Ishii-Ueda_CCDM},
the number of intersections of $y$ with $r$
coincides with that of $y$ with $p_i \equiv q_i$, which is $1$.
If $r$ passes through $a$,
the direction of the intersection with $y$ is different
from that of the intersection of $q_i$ with $y$,
which is a contradiction.
This shows $a \in D_2$,
and Proposition \ref{prop:non-degenerate} is proved.
\end{proof}

\begin{definition}
For a pair $(z, w)$ of zigzag paths with adjacent slopes,
the perfect matching
obtained as in the proof of Proposition \ref{prop:non-degenerate}
containing the edge at the intersection is said to
{\em come from a pair of zigzag paths with adjacent slopes}.
\end{definition}


\begin{figure}[htbp]
\centering
\input{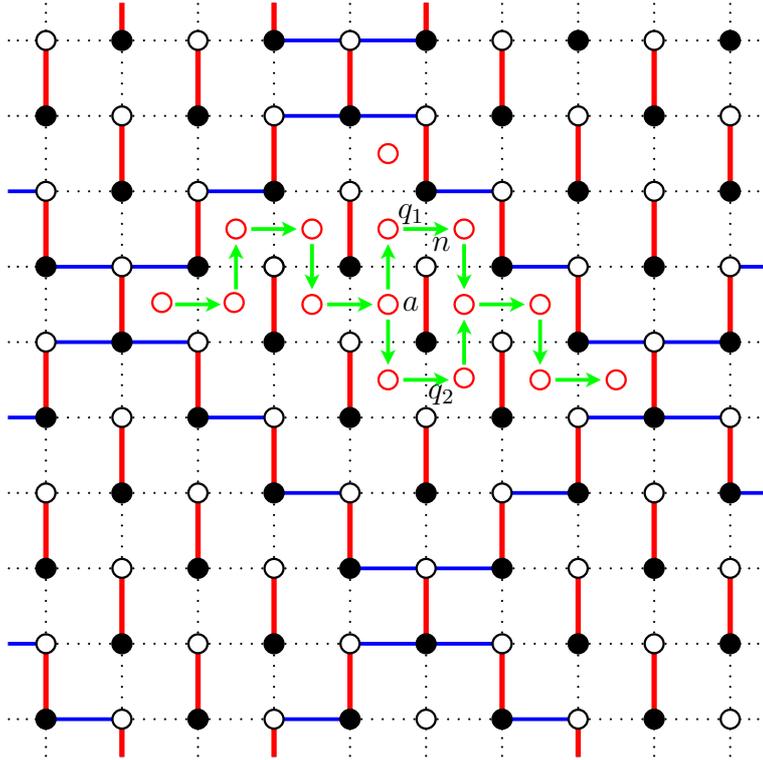}
\caption{The perfect matching
associated with a pair of zigzag paths}
\label{fg:large_square_tile_zigzag_pm}
\end{figure}

Recall from Section \ref{sc:pm_moduli}
that a path $p$ on a quiver is {\em allowed}
by a perfect matching $D$
if the path $a$ does not contain any arrow
dual to an edge in $D$.
The proof of Proposition \ref{prop:non-degenerate}
also shows the following:

\begin{lemma}
Let $(z, w)$ be a pair of zigzag paths
with adjacent slopes and
$D$ be the corresponding perfect matching.
Then for any large hexagon,
one has the following:
\begin{enumerate}
\item
For any vertex $v$ in the large hexagon,
there is a path allowed by $D$ inside the large hexagon
from the source vertex
to the vertex $v$.
\item
For any vertex $v$ in the large hexagon,
there is a path allowed by $D$ inside the large hexagon
from the vertex $v$
to the sink vertex.
\end{enumerate}
\end{lemma}

\section{Corner perfect matchings} \label{sc:corner}

In this section,
we prove the following
as an application of large hexagons:

\begin{proposition} \label{prop:large_hexagon}
Let $G$ be a consistent dimer model.
Then any choice
of a pair of zigzag paths with adjacent slopes
determines a corner $\frakc$
of the characteristic polygon $\Delta$
and induces the following:
\begin{enumerate}
 \item
The division of the torus $T = \bR^2 / \bZ^2$
into large hexagons.
 \item
A functor
$$
 \phi_\frakc^* :  \module (\bC[x,y,z]  \rtimes A)
  \to \module \bC \Gamma
$$
from the category of representations of the McKay quiver
of some finite small abelian subgroup
$A \subset \GL_2(\bC) \subset \SL_3(\bC)$
to that of the path algebra of the quiver $\Gamma$ with relations
associated with the dimer model $G$.
 \item
An embedding
$$
 \varphi_\frakc : \ahilb(\bC^3) \hookrightarrow \scM_\theta
$$
of the $A$-Hilbert scheme
as an open subscheme
of the moduli space $\scM_{\theta}$
for some generic stability parameter $\theta$.
\end{enumerate}
\end{proposition}

We also show the following characterization
of corner perfect matchings in this section:

\begin{proposition} \label{prop:corner_perfect_matching}
The following are equivalent
for a perfect matching $D$
in a consistent dimer model:
\begin{enumerate}
 \item
$D$ is simple.
 \item
$D$ is multiplicity free.
 \item
$D$ is a corner perfect matching.
 \item
$D$ comes from a pair of zigzag paths with adjacent slopes.
\end{enumerate}
\end{proposition}

We first prove
Proposition \ref{prop:corner_perfect_matching}.
The proof is divided into four steps:

\begin{step} \label{st:corner_zigzag}
A perfect matching is a corner perfect matching
if and only if it comes
from a pair of zigzag paths with adjacent slopes.
\end{step}

\begin{proof}
The if part follows from the fact
that the height change of a perfect matching
coming from a pair of zigzag paths with adjacent slopes
satisfies the equality in the inequality \eqref{eq:zigzag}
coming from both of these zigzag paths.

To show the only if part, consider three zigzag paths
$z_1, z_2, z_3$ with consecutive slopes.
Let $D_1$ and $D_2$ be the perfect matchings coming from
$z_1, z_2$ and $z_2, z_3$ respectively.
Then \eqref{eq:zigzag} implies
$ \langle h(D, D_1), [z_2] \rangle \le 0$
for any $D$, where the equality holds for $D=D_1, D_2$.
This shows that
the line segment connecting the height changes of the corner perfect
matchings $D_1$ and $D_2$ is on the boundary of the Newton polygon.
In this way, we see that every corner perfect matching comes from a pair
of zigzag paths with adjacent slopes.
\end{proof}

\begin{step} \label{st:pm_simple}
A perfect matching
coming from a pair of zigzag paths with adjacent slopes
is simple.
\end{step}

\begin{proof}
We have to show that
the corresponding quiver representation $M$ is simple,
i.e., has no non-trivial submodule.
This follows from the fact
that in a perfect matching
coming from a pair of zigzag paths with adjacent slopes,
one can find an allowed path
from any vertex to any other vertex in the quiver.
Indeed, starting from any vertex,
one can first go to the sink of the large hexagon $h_1$
where the vertex belongs,
and then to the adjacent vertex which is the source of adjacent large hexagon $h_2$
by the path going around one of the nodes on the edge
separating $h_1$ and $h_2$.
Recall that one can go from the source of a large hexagon
to any other vertex in the same large hexagon
only through an allowed path.
Note also that one can go
from the source of one large hexagon
to the source of another large hexagon
adjacent in the $x$- and $y$-direction.
Since one can go from one large hexagon
to any other large hexagon
by multiplying sufficiently many $x$ and $y$,
Step \ref{st:pm_simple} is proved.
\end{proof}

\begin{step}
A simple perfect matching is a corner perfect matching.
\end{step}

\begin{proof}
Since simple modules are $\theta$-stable
for any $\theta$, the divisor corresponding to
a simple perfect matching is not contracted in
the affine quotient $\scMbar_0$.
Hence it must be a corner perfect matching.
\end{proof}

\begin{step}
A perfect matching is multiplicity-free
if and only if it is simple.
\end{step}

\begin{proof}
Let us first prove the only if part:
Assume $M$ has a non-trivial submodule.
Then one can find a stability parameter $\theta$
such that $M$ is not $\theta$-semistable.
Since $M$ is $0$-semistable and
the map $\scM_\theta \to \scMbar_0$ is projective,
there is another $\theta$-semistable representation $N$
with the same height change.

Now we prove the if part:
Assume that $M$ is simple and take any module $N$
with the same height change as $M$.
Choose a stability parameter $\theta$ such that
semistability implies stability and $N$ is $\theta$-stable
\cite[Lemma 6.2]{Ishii-Ueda_08}.
Since $M$ is also $\theta$-stable with the same height change as $N$,
the modules $N$ and $M$ must belong to the same $\bT$-orbit,
so that the corresponding perfect matchings are identical.
\end{proof}



This completes the proof of
Proposition \ref{prop:corner_perfect_matching}.
The proof of Step \ref{st:corner_zigzag} also shows the following:

\begin{corollary}\label{cor:zigzag}
The set of slopes of zigzag paths in a consistent dimer model
is in one-to-one correspondence
with the set of sides of the characteristic polygon,
so that each slope is normal to the corresponding side.
\end{corollary}

Let $A \subset \GL(2, \bC)$ be the finite small subgroup
whose McKay quiver $\Lambda$ is identified
with the tessellation by large hexagons
as in Section \ref{sc:large-McKay}.
We discuss the embedding of $\ahilb(\bC^3)$ into $\scM_{\theta}$
for a suitable choice of $\theta$.
Let $D$ be the perfect matching
coming from a pair of zigzag paths with adjacent slopes.
We regard quivers as categories
as in Section \ref{sc:quiver-category},
and define a functor
$$
 \phi_\frakc : \Gamma \to \Lambda
$$
as follows:
\begin{itemize}
\item
A vertex of $\Gamma$ is sent to the large hexagon containing it.
 \item
An arrow inside a large hexagon that is not contained in $D$
goes to the identity of the large hexagon.
 \item
An arrow inside a large hexagon that is contained in $D$
goes to the small cycle at the large hexagon.
 \item
Suppose an arrow $a$ of $\Gamma$ is on the boundary of two large hexagons.
Let $b$ be the arrow of $\Lambda$ connecting the two large hexagons.
If $a$ is in the same direction as $b$, then $a$ goes to $b$.
If $a$ is in the opposite direction, then $a$ goes to the path of
length two that connects the two large hexagons
in the same direction as $a$.
\end{itemize}
Recall that a representation of a quiver is regarded as a functor
from the quiver as a category to the category of vector spaces.
Thus $\phi_\frakc$ induces a functor
$$
 \phi_\frakc^*: \module \bC \Lambda \to \module \bC \Gamma.
$$
The functor $\phi_\frakc^*$ sends a $G$-cluster
to a representation of $\Gamma$
with dimension vector $(1, \dots, 1)$.

Let $h_0$ be the large hexagon
identified with the trivial representation
in the McKay quiver $\Lambda$ of $A$.
Choose a parameter
$
 \eta \in \Hom(\bZ^V, \bQ)
$
satisfying the following:
\begin{itemize}
 \item
If a vertex $v$ is not the source of a large hexagon, then $\eta(v)=1$.
 \item
The sum of $\eta(v)$ inside a fixed large hexagon is $0$.
\end{itemize}
Then take a sufficiently small $\epsilon >0$ and
define a stability parameter $\theta \in \Hom(\bZ^V, \bQ)$ as follows:
\begin{itemize}
 \item
If $v$ is the source of a large hexagon other than $h_0$,
then $\theta(v)=\eta(v)+\epsilon$.
 \item
If $v$ is the source of $h_0$,
then $\theta(v)= \eta(v) - (\# A -1)\epsilon$.
 \item
For other vertices $v$, we set $\theta(v)=\eta(v)$.
\end{itemize}
One can easily see that every $A$-cluster goes
to a $\theta$-stable representation of $\Gamma$.
This gives an open immersion
$\ahilb(\bC^3) \to \scM_{\theta}$, and
Proposition \ref{prop:large_hexagon} is proved.

\section{Description of the algorithm}
 \label{sc:algorithm}


\subsection{Removal of edges}
 \label{ss:algorithm_edges}

Let $G = (B, W, E)$ be a consistent dimer model.
The algorithm to remove the corner $\frakc$
from the characteristic polygon $\Delta$ is the following:

\begin{enumerate}
 \item[0.]
 Remove all divalent nodes. 
This step in fact is not necessary
but simplifies the exposition below.
 \item \label{st:zigzag}
Choose a pair of zigzag paths with adjacent slopes
corresponding to the corner $\frakc$.
 \item \label{st:origin}
Choose an identification of the resulting large hexagons
with vertices of the McKay quiver
for a finite small abelian group
$
 A \subset \GL_2(\bC) \subset \SL_3(\bC)
$
by choosing the large hexagon
corresponding to the trivial representation.
 \item \label{st:removal}
Remove the edges of the dimer
corresponding to the arrows of the quiver
going from the sources of the large hexagons
corresponding to special representations
to the sinks of the adjacent large hexagons
related by ``multiplication by $z$''.
\end{enumerate}

One has a choice
in Steps \ref{st:zigzag} and \ref{st:origin},
and the result of the operation depends on this choice.
See Section \ref{ss:algorithm_examples}
below for examples.

\subsection{Inversion of arrows}
 \label{ss:algorithm_arrows}

Removing an edge of a dimer model corresponds to merging  adjacent vertices into a single vertex.
It also corresponds to adding an inverse arrow under a mild condition,
which is always satisfied when we remove a corner
from the characteristic polygon
of a consistent dimer model:

\begin{lemma} \label{lemma:operation-dimer}
Let $G=(B, W, E)$ be a dimer model without divalent nodes, and
$\Gamma$ be the associated quiver with relations.
Let $S$ be a subset of $E$, and
assume that every node is contained in at least two edges in $E \setminus S$
and that there is a pair of linearly independent cycles on $T$ consisting of
edges in $E\setminus S$. 
\begin{enumerate}
\item
If there are no (not necessarily oriented) cycles of $\Gamma$
consisting of arrows in $S$,
then $G'=(B, W, E\setminus S)$ is again a dimer model,
and the path algebra associated with $G'$ is Morita equivalent to the
path algebra of the quiver with relations obtained from $\Gamma$ by adding the inverses $a^{-1}$ of arrows $a \in S$ together with relations $a a^{-1} = e_{t(a)}$ and $a^{-1}a = e_{s(a)}$.
Here $e_v$ is the idempotent element
associated with a vertex $v$ of a quiver.
\item
If $G'$ is a dimer model,
then there are no cycles consisting of arrows in $S$.
\end{enumerate}
\end{lemma}

\begin{proof}
The assumption in 1 implies that connected components of $T \setminus \bigcup_{e \in E\setminus S} e$ are simply connected, which ensures that $G'$ is a dimer model.
It is easy to see that the categories of representations of the above two quivers with relations are equivalent to each other.

For 2, assume that there is a cycle consisting of arrows in $S$.
Then the connected component of $T \setminus \bigcup_{e \in E\setminus S} e$ containing the cycle is not simply connected
or contains an isolated node, which implies that $G'$ is not a dimer model.
\end{proof}

We show
in Section \ref{sc:preservation-of-consistency}
that the consistency condition is preserved under the operation
in Section \ref{ss:algorithm_edges},
so that Lemma \ref{lemma:operation-dimer} can be applied.
The point of view of adding inverse arrows will be used
in Sections 
\ref{sc:surj-G-Hilb},
\ref{sc:tilting-general}, and
\ref{sc:surj-general}
to prove the derived equivalence inductively.

\subsection{Examples}
 \label{ss:algorithm_examples}

As an example,
consider the construction of dimer models
for the hexagon in Figure \ref{fg:dp3_diagram}
starting from the dimer model in Figure \ref{fg:2x2}
corresponding to the square lattice polygon
in Figure \ref{fg:2x2_diagram}
by removing two vertices.

\begin{figure}[htbp]
\begin{minipage}{.3 \linewidth}
\centering
\input{2x2_diagram.pst}
\caption{A square}
\label{fg:2x2_diagram}
\end{minipage}
\begin{minipage}{.3 \linewidth}
\centering
\input{2x2-1_diagram.pst}
\caption{A pentagon}
\label{fg:2x2-1_diagram}
\end{minipage}
\begin{minipage}{.3 \linewidth}
\centering
\input{dp3_diagram.pst}
\caption{A hexagon}
\label{fg:dp3_diagram}
\end{minipage}
\end{figure}

\begin{figure}[htbp]
\centering
\input{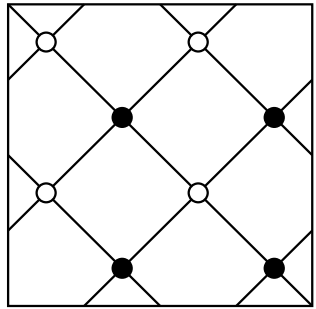}
\caption{A dimer model for the square lattice polygon}
\label{fg:2x2}
\end{figure}

To remove the top left corner from the square lattice polygon
in Figure \ref{fg:2x2_diagram},
we have to choose a pair of zigzag paths,
one from each of those with homology classes
$(-1, 0)$ (shown in red in Figure \ref{fg:2x2_zigzag})
and $(0, 1)$ (shown in blue in Figure \ref{fg:2x2_zigzag}).
There are four choices in Step \ref{st:zigzag},
which actually do not matter for symmetry reasons.
There is no choice in Step \ref{st:origin},
and Figure \ref{fg:2x2-1} shows the resulting dimer model.

\begin{figure}[htbp]
\centering
\input{2x2_zigzag.pst}
\caption{zigzag paths}
\label{fg:2x2_zigzag}
\end{figure}

\begin{figure}[htbp]
\centering
\input{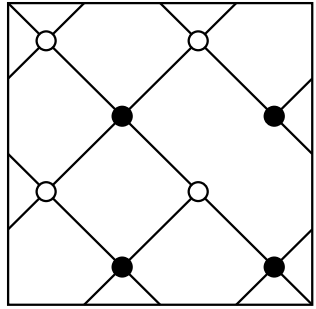}
\caption{The dimer model for the pentagonal lattice polygon}
\label{fg:2x2-1}
\end{figure}

\begin{figure}[htbp]
\centering
\input{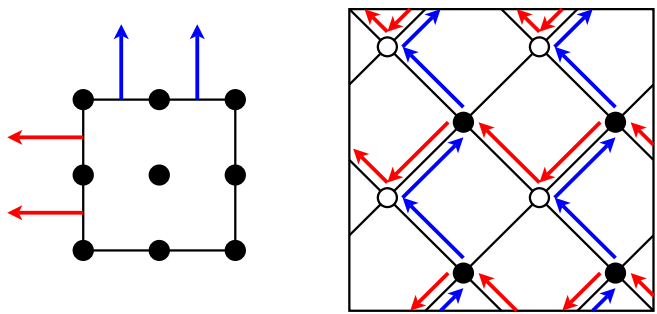}
\caption{zigzag paths}
\label{fg:2x2-1_zigzag}
\end{figure}

Now consider the the removal of the lower-right corner
from the pentagonal lattice polygon in Figure \ref{fg:2x2-1_diagram}.
In this case there are four choices in Step \ref{st:zigzag},
which lead to the dimer models
shown in Figure \ref{fg:2x2-2}.
Note that the dimer models 2 and 4 are obtained
from the dimer models 1 and 3 respectively
by changing the colors of the nodes,
so that the corresponding quivers are related
by the reversal of arrows.

\begin{figure}[htbp]
\centering
\input{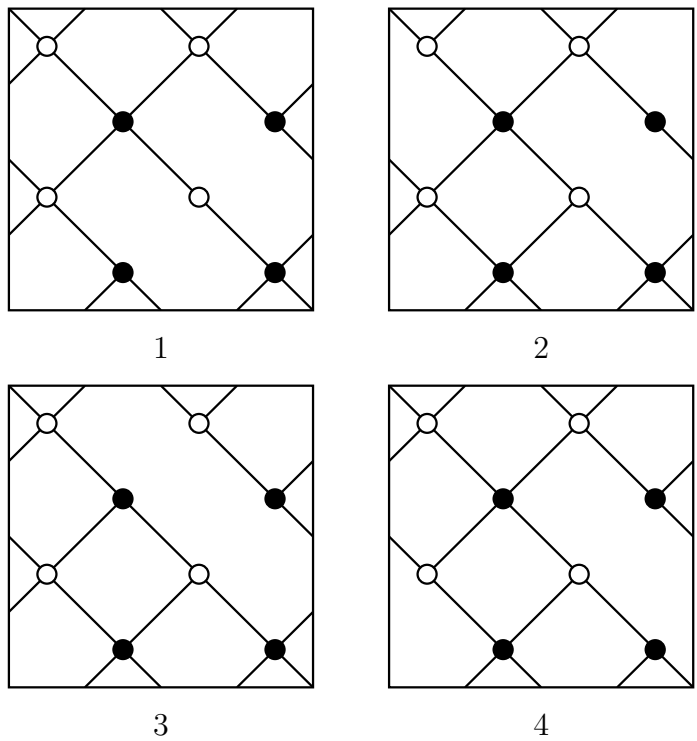}
\caption{Dimer models for the hexagonal lattice polygon}
\label{fg:2x2-2}
\end{figure}

The dimer model 1 in Figure \ref{fg:2x2-2}
has a divalent white node,
and one obtains the dimer model in Figure \ref{fg:hexagon_graph1}
by removing it.
The dimer model 3 is equivalent to the dimer model
shown in Figure \ref{fg:hexagon_graph2}.

\begin{figure}[htbp]
\centering
\input{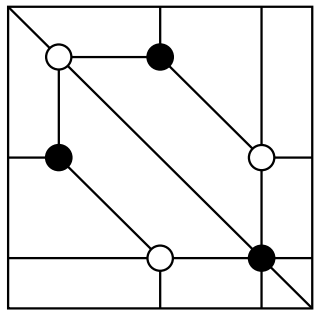}
\caption{A dimer model equivalent to the dimer model 1
in Figure \ref{fg:2x2-2}}
\label{fg:hexagon_graph1}
\end{figure}

\begin{figure}[htbp]
\centering
\input{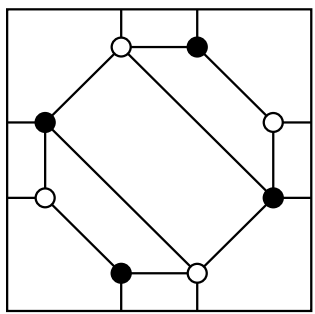}
\caption{A dimer model equivalent to the dimer model 3
in Figure \ref{fg:2x2-2}}
\label{fg:hexagon_graph2}
\end{figure}

The zigzag paths on the dimer model in Figure \ref{fg:hexagon_graph1}
are shown in Figure \ref{fg:hexagon_zigzag1}.

\begin{figure}[htbp]
\centering
\input{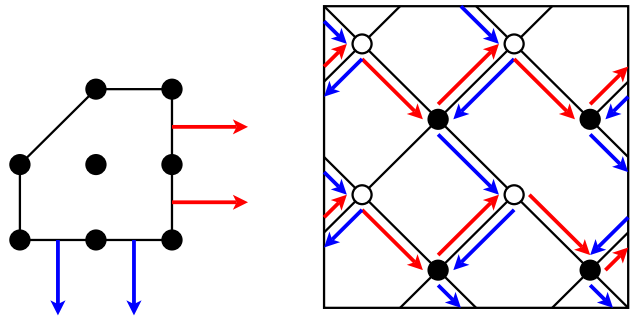}
\caption{zigzag paths}
\label{fg:hexagon_zigzag1}
\end{figure}

From the dimer model in Figure \ref{fg:hexagon_graph1},
one can construct the dimer model for $\bP^2$
by removing three vertices from the lattice polygon
as in Figure \ref{fg:hexagon_triangle}.

\begin{figure}[htbp]
\centering
\input{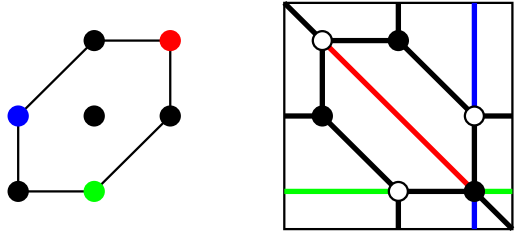}
\caption{From a hexagon to a triangle}
\label{fg:hexagon_triangle}
\end{figure}

Similarly,
from the dimer model in Figure \ref{fg:hexagon_graph1},
one can construct the dimer model for $\bP^1 \times \bP^1$
by removing two vertices from the lattice polygon
as in Figure \ref{fg:hexagon_square}.

\begin{figure}[htbp]
\centering
\input{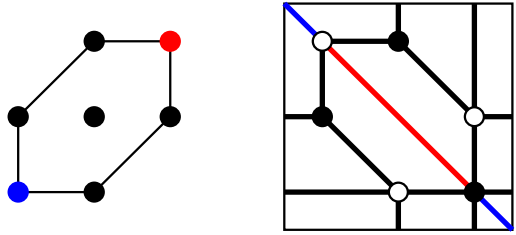}
\caption{From a hexagon to a square}
\label{fg:hexagon_square}
\end{figure}

Next we discuss a simple example
where the special McKay correspondence plays a role.
Let $A = \la \frac{1}{5}(1,2) \ra$
be the subgroup of $\GL_2(\bC)$
generated by $\diag(\zeta, \zeta^2)$
for $\zeta = \exp(2 \pi \sqrt{-1}/5)$.
Recall from Section \ref{sc:continued-fraction}
that the integers
$r$, $b_1, \dots, b_r$ and $i_0, \dots, i_{r+1}$
are defined inductively by
$i_0 := n$, $i_1:=q$, and
$$
 i_t = b_{t+1} i_{t+1} - i_{t+2} \quad (0 < i_{t+2} < i_{t+1})
$$
until we finally obtain $i_r=1$ and $i_{r+1}=0$.
This gives
\begin{align*}
 5 &= 3 \cdot 2 - 1, \\
 2 &= 2 \cdot 1 - 0,
\end{align*}
so that $r = 2$, $(b_1, b_2) = (3, 2)$, and
$(i_0, i_1, i_2) = (5, 2, 1)$.
The continued fraction expansion
\eqref{eq:continued_fraction}
is given by
$$
 \frac{n}{q}
  = \frac{5}{2}
  = b_1 - \cfrac{1}{b_2 - \cfrac{1}{\ddots -\cfrac{1}{b_r}}}
  = 3 - \frac{1}{2},
$$
and the special representations are given
by $\rho_5 = \rho_0$, $\rho_1$ and $\rho_2$.
The McKay quiver for $A$
as a subgroup of $\SL_3(\bC)$ is the quiver
associated with the dimer model
shown in Figure \ref{fg:Z5_graph},
where the parallelogram shows
a fundamental region of the torus.
To remove the top right corner
from the characteristic polygon
shown in Figure \ref{fg:Z5_diagram},
we have to remove edges corresponding to
`multiplication by $z$'
from special representations.
These edges are shown in dotted lines
in Figure \ref{fg:Z5_graph},
and by removing them,
one obtains the dimer model
shown in Figure \ref{fg:Z2_graph1}.
This dimer model contains divalent nodes,
and by removing them,
one obtains the dimer model
shown in Figure \ref{fg:Z2_graph2},
which is exactly the dimer model
corresponding to the characteristic polygon
shown in Figure \ref{fg:Z2_diagram}.

\begin{figure}[htbp]
\begin{minipage}{.5 \textwidth}
\centering
\input{Z5_diagram.pst}
\caption{The characteristic polygon}
\label{fg:Z5_diagram}
\end{minipage}
\begin{minipage}{.5 \textwidth}
\centering
\input{Z2_diagram.pst}
\caption{The characteristic polygon}
\label{fg:Z2_diagram}
\end{minipage}
\end{figure}
\begin{figure}[htbp]
\centering
\input{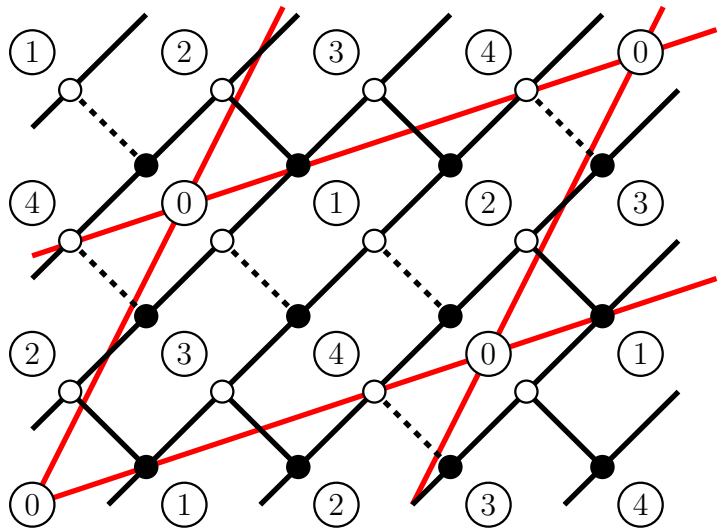}
\caption{The dimer model
associated with $\Delta$}
\label{fg:Z5_graph}
\end{figure}
\begin{figure}[htbp]
\centering
\input{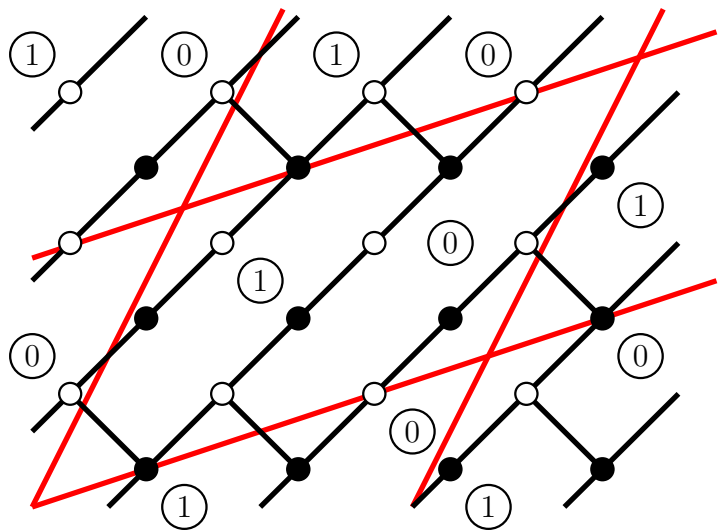}
\caption{The dimer model
after the operation}
\label{fg:Z2_graph1}
\end{figure}
\begin{figure}[htbp]
\centering
\input{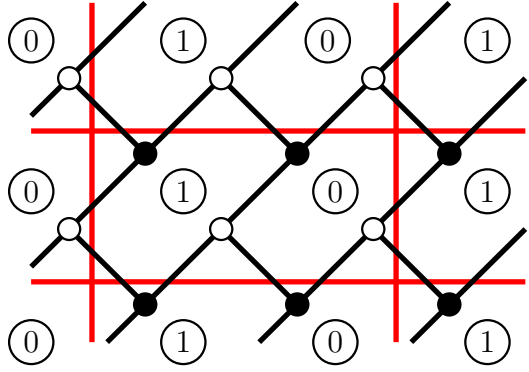}
\caption{The dimer model
after removing divalent nodes}
\label{fg:Z2_graph2}
\end{figure}

\section{Preservation of the consistency}
 \label{sc:preservation-of-consistency}

We use the same notation as in Section \ref{sc:continued-fraction}.
We prove the following in this section:

\begin{proposition} \label{prop:preservation-of-consistency}
A consistent dimer model remains consistent
after the operation described in Section \ref{sc:algorithm},
if the lattice points of the polygon other than the removed one
do not lie on a line.
\end{proposition}

We need the following lemma
to prove Proposition \ref{prop:preservation-of-consistency}:

\begin{lemma} \label{lm:no_special}
Let $t \in [1, r+1]$, $a \in (0, i_{t-1}-i_t)$ and
$b \in (0, j_t-j_{t-1})$ be integers.
Then $i_t + a + bq$ is special if and only if $a=b=0$.
\end{lemma}

\begin{proof}
Write
$$
a = d_t i_t + d_{t+1} i_{t+1} + \dots + d_r i_r
$$
as in Theorem \ref{theorem:wunram}.
Using the same theorem for the dual sequence, we can write
$$
b = d_{t-1} j_{t-1} + d_{t-2} j_{t-2} + \dots + d_1j_1.
$$
Then Lemma \ref{lm:wunram_dual} implies
$$
i_t + a + bq \equiv d_1 i_1 + \dots + d_{t-1}i_{t-1} + (d_t+1)i_t +d_{t+1} i_{t+1} + \dots +d_r i_r.
$$
Therefore if the sequence
$
 (d_1, \dots, d_{t-1}, d_t +1, d_{t+1}, \dots, d_r)
$
satisfies the condition in Lemma \ref{lemma:wunramvanishing},
then $i_t + a + bq$ is special if and only if $d_1=\dots=d_r=0$ by the uniqueness
of the expression in Theorem \ref{theorem:wunram}.

By using
$
 b_t i_t = i_{t-1} + i_{t+1}
$
and the assumption
$
 a < i_{t-1}-i_t,
$
we obtain
$$
 (d_t +1 -b_t)i_t + (d_{t+1} +1)i_{t+1} + d_{t+2}i_{t+2} +
   \dots +d_r i_r < 0
$$
which implies $d_t \le b_t -2$.
Moreover,
if the equality $d_t = b_t -2$ holds, then we have
$$
d_{t+1} i_{t+1} + d_{t+2}i_{t+2}+ \dots +d_r i_r < i_t - i_{t+1},
$$
which is of the same form as the assumption $a < i_{t-1}-i_t$ with $t$ increased by $1$
so that we obtain $d_{t+1} \le b_{t+1}-2$.
Thus we can inductively show that if $d_k = b_k -1$ for some $k > t$,
then there is an integer $l \in (t, k)$
with $d_l \le b_l -3$.

We can argue in the same way to conclude:
$d_{t-1} \le b_{t-1} -2$ and
if $d_k = b_k -1$ for some $k < t-1$,
then there is an integer
$l \in (k, t-1)$ with $d_l \le b_l -3$.

Thus we have shown that the sequence
$
 (d_1, \dots, d_{t-1}, d_t+1, d_{t+1}, \dots, d_r)
$
satisfies the condition in Lemma \ref{lemma:wunramvanishing}.
\end{proof}

Now we prove Proposition \ref{prop:preservation-of-consistency}:

\begin{proof}[Proof of
Proposition \ref{prop:preservation-of-consistency}]

We first note that
if the zigzag paths of the bicolored graph
obtained by the operation satisfy the consistency condition,
then the assumption
implies that the bicolored graph satisfies the condition in Lemma \ref{lemma:operation-dimer}
and hence is actually a dimer model.
We prove the consistency conditions in two steps.

\setcounter{step}{0}
\begin{step}
The case
$
 \ahilb(\bC^3) \setminus \ahilb(\bC^2)
$
for $A=\langle\frac{1}{n}(1,q)\rangle \subset \GL_2(\bC)$.
\end{step}

Let $\Lambda$ be the hexagonal dimer model for $\ahilb(\bC^3)$.
The associated quiver is the McKay quiver of $A$
where the vertices are the irreducible representations of $A$ and
there are three arrows from each vertex corresponding to the multiplications by the coordinate
functions $x, y, z$.
Regard the set $V$ of vertices as
$V=(\bZ/n\bZ)^*=\bZ/n\bZ$
and let $\alpha_i, \beta_i, \gamma_i$ be the three arrows
with the source $i\in V$ whose targets are $i+1$, $i+q$, $i-q-1$ respectively.
A zigzag path of $\Lambda$ is of one of the following three forms according to its homology class:
$(\dots, \beta_{i+1}, \alpha_i, \beta_{i-q}, \alpha_{i-q-1},\dots)$,
$(\dots, \gamma_{i+q}, \beta_i, \gamma_{i+q+1}, \beta_{i+1},\dots)$ or
$(\dots, \alpha_{i-q-1}, \gamma_i, \alpha_{i-1}, \gamma_{i+q}, \dots)$.

Let $\Lambda'$ be the bicolored graph
obtained from $\Lambda$ by the operation in Section \ref{sc:algorithm}
(i.e., by removing the edges $\gamma_i$'s for special $i$'s).
Of three homology classes of zigzag paths on $\Lambda$,
only the ones consisting of $\alpha$'s and $\beta$'s
survive in $\Lambda'$.
Other two zigzag paths will be transformed into
new zigzag paths on $\Lambda'$,
indexed by $t$ with $t \in \bZ/r\bZ$ as follows:
Start with the edge $\beta_{i_{t-1}-q}$ whose target is the special hexagon $i_{t-1}$, 
next choose the adjacent edge $\gamma_{i_{t-1}+1}$ if $i_{t-1}+1$ is not special,
and
go along the old zigzag path consisting of $\beta$'s and $\gamma$'s
until one arrives at the next special hexagon $i_t$,
where one is blocked by the removed edge $\gamma_{i_t}$.
Then one changes the direction and
go along the old zigzag path consisting of $\alpha$'s and $\gamma$'s.
By virtue of \eqref{equation:special},
one comes back to the starting point
without meeting any other removed edges.

Now let us check the consistency of the new dimer model.
It is obvious that a new zigzag path has no self-intersection
on the universal cover.
Choose two zigzag paths on $\Lambda'$.
If they are both old, i.e., zigzag paths of $\Lambda$, then they do not meet at all.
If one is old and the other is new,
then they meet more than once in general
but always in the opposite direction.
If they are both new,
then they meet at most once on the universal cover,
since there are no special representations
in the rectangular region in \pref{fg:pres_cons}
by \pref{lm:no_special}.

\begin{figure}[ht]
\centering
\input{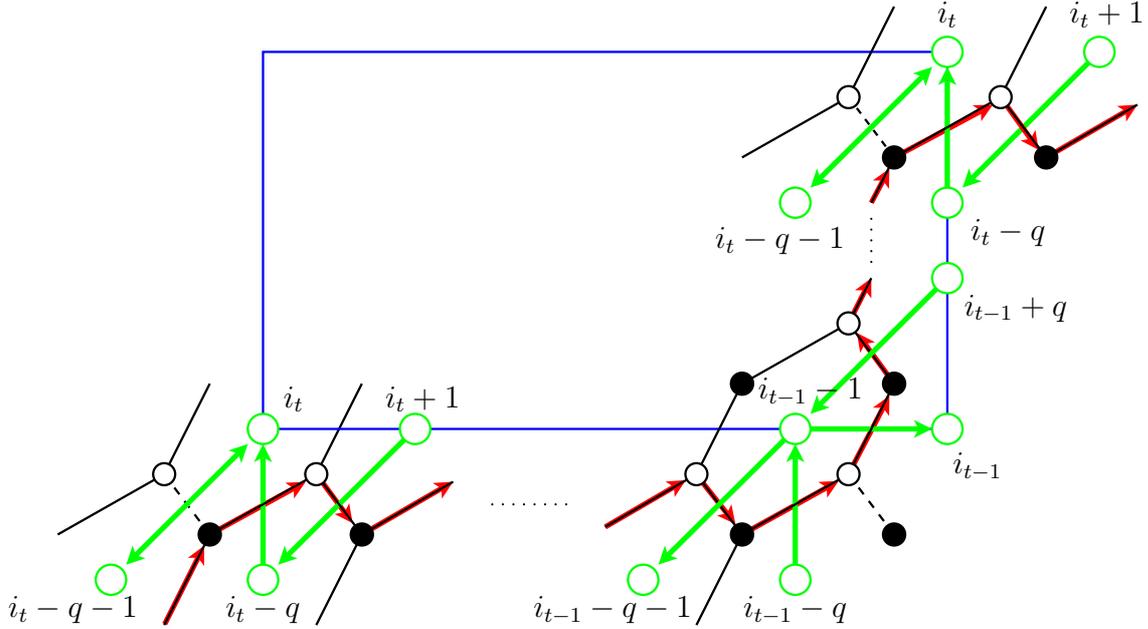}
\caption{New zigzag paths intersect at most once on the universal cover}
\label{fg:pres_cons}
\end{figure}


\begin{step}
The general case.
\end{step}
Old zigzag paths except the chosen two survive,
and new zigzag paths are described
in the same way as above using the large hexagons.
Let us analyze intersections of two zigzag paths
in the new dimer model.
If two zigzag paths are both old or new,
then the same reasoning as step 1 shows that
they do not intersect in the same direction twice.
Take one new zigzag path and one survivor from the old one,
and suppose they meet twice in the same direction.
Since we have chosen two zigzag paths with adjacent slopes
to perform the operation,
the slope of the survivor cannot be
in between the slopes of these two zigzag paths.
This implies that the survivor must meet
either of the two zigzag paths twice in the same direction,
thus contradicting the consistency of the old dimer model.
\end{proof}

\section{Zigzag paths and characteristic polygons}
 \label{sc:zigzag-polygon}

We use the relation
between zigzag paths and
the characteristic polygon
to show that the characteristic polygon
changes as expected
under the operation.

Let $([z_i])_{i=1}^k$ be the sequence of slopes of zigzag paths
ordered cyclically starting from any zigzag path.
Here $k$ is the number of zigzag paths, and
some of the slopes may coincide in general.
Define another sequence $(w_i)_{i=1}^r$ in $\bZ^2$ by
$w_0 = 0$ and 
$$
 w_{i+1} = w_i + [z_{i+1}]', \qquad i = 0, \dots, k-1
$$
where $[z_{i+1}]'$ is obtained from $[z_{i+1}]$ by
rotating by 90 degrees in the positive direction.
Note that one has
$
 w_r = 0
$
since every edge is contained in exactly two zigzag paths
whose directions on that edge is opposite,
and hence the homology classes of the zigzag paths add up to zero.
The convex hull of $(w_i)_{i=1}^r$ is called
the {\em zigzag polygon}.

The following theorem is proved by Gulotta
\cite[Theorem 3.3]{Gulotta}
for properly-ordered dimer models:

\begin{theorem}
 \label{th:gulotta}
For a consistent dimer model,
the characteristic polygon $\Delta$ coincides with the zigzag polygon
up to translation.
\end{theorem}

\begin{proof}
We already have Corollary \ref{cor:zigzag},
and it suffices to show that
the number of zigzag paths with a given slope
coincides with the number of the primitive side segments
on the corresponding side of $\Delta$.
Let $(z, w)$ be a pair of zigzag paths with adjacent slopes, and
$D = D_0$ be the corner perfect matching coming from $(z, w)$.
Let $\{ z=z_1, \dots, z_r \}$ be the set of zigzag paths with slope $[z]$.
Recall that $D$ contains half of the edges constituting $z$.
On the other hand,
if we construct a corner perfect matching by using $z_i$ and $w$,
then it must coincide with $D$ since $D$ is multiplicity-free. 
Thus $D$ contains half of the edges constituting $z_i$ for each $i$.
Let $D_1$ be the perfect matching such that the symmetric difference
$D_0 \vartriangle D_1 :=D_0 \cup D_1 \setminus (D_0 \cap D_1)$ is $z_1$.
Similarly,
let $D_i$ be the perfect matching with $D_{i-1} \vartriangle D_i = z_i$
for $i=2,\dots, r$. 
Then $D = D_0, D_1, \dots, D_r$ lie
on the side of $\Delta$ perpendicular to $[z]$.
Let $D'$ be the other corner perfect matching on this side of $\Delta$.
To see that the numbers coincide,
it suffices to show that $D' = D_r$.

The symmetric difference $D \vartriangle D'$,
equipped with the black-to-white flow on $D$ and
the white-to-black flow on $D'$,
gives a cycle on the torus $T$.
Recall from Section \ref{sc:pm_moduli}
that one can associate a representation of $\Gamma$
to a perfect matching.
If there is a homologically trivial cycle in $D \vartriangle D'$,
then the vertices surrounded by this cycle
constitute a sub-representation or a quotient representation
of the representation associated with $D$ or $D'$.
A cycle in the homology class $-[z]$ leads
to a sub-representation or a quotient representation
in a similar way.
On the other hand,
corner perfect matchings are simple
by Proposition \ref{prop:corner_perfect_matching}.
This shows that $D \vartriangle D'$ consists
of disjoint cycles in the class $[z]$.
It follows from the definition of the height change
that the number of cycles in $D \vartriangle D'$ coincides
with the number of primitive segments
on the side connecting $D$ and $D'$.
In particular, $D' \cap D_i \cap z_i$ is empty for each $i$.

Now we perform the operation in Section \ref{sc:algorithm}
for the pair $(z, w)$.
Then $D_1, \dots, D_r, D'$ survive as perfect matchings of the new
consistent dimer model with the new corner perfect matching $D_1$.
The zigzag paths of slope $[z]$ in the new dimer model are
$z_2, \dots, z_r$.
We can repeat the operation until $D_r$ becomes a corner perfect matching.
Then, there are no zigzag paths of slope $[z]$, which implies $D_r=D'$
by Corollary \ref{cor:zigzag}.
\end{proof}

Theorem \ref{th:gulotta} yields
the second part of Theorem \ref{th:removal}:

\begin{proposition} \label{prop:polygon}
The characteristic polygon of the dimer model
after the operation in Section \ref{sc:algorithm}
is obtained by removing the chosen corner
and taking the convex hull of the rest.
\end{proposition}

\begin{proof}
In the proof of Proposition \ref{prop:preservation-of-consistency},
we have described the change of  zigzag paths
under the operation.
Theorem \ref{th:gulotta} shows that this induces
the desired change in the characteristic polygon.
\end{proof}

Theorem \ref{th:gulotta} and
Lemma \ref{lm:zigzag-adjacency}
gives the following uniqueness result
in the case of lattice triangles:

\begin{proposition} \label{prop:triangle}
For any lattice triangle $\Delta$,
there is a unique consistent dimer model
whose characteristic polygon coincides with $\Delta$.
\end{proposition}

\begin{proof}
In the case of a triangle,
any pair of zigzag paths are adjacent
or have the same slope.
If they have the same slope,
then consistency condition prevents then
from intersecting at all.
If they are adjacent,
then they can intersect only once
on the universal cover
by Lemma \ref{lm:zigzag-adjacency}.
These two conditions suffice to show that
the resulting dimer model gives a hexagonal tiling
of the 2-torus,
and the corresponding quiver is the McKay quiver
for some abelian subgroup of $\SL_3(\bC)$.
\end{proof}

See also \cite[Theorem 1.2]{Ueda-Yamazaki_NBTMQ}
for a closely-related uniqueness result.
The uniqueness fails even for squares;
see e.g. \cite{Ueda-Yamazaki_BTP}
for a discussion of an example.

Another corollary is the following statement,
which is stronger than Proposition
\ref{prop:non-degenerate}:

\begin{corollary}\label{cor:corner}
Let $G$ be a consistent dimer model.
Then for any edge of $G$,
there is a corner perfect matching containing it.
\end{corollary}

\begin{proof}
For an edge $e$, choose a zigzag path $z$ containing $e$.
Then we can construct a corner perfect matching $D$ of $G$
which contains half of the edges of $z$ as in the proof of
Proposition \ref{prop:non-degenerate}.
If $D$ contains $e$, we are done.
If $D$ dones not contain $e$,
then the other corner perfect matching $D'$
in the proof of Theorem \ref{th:gulotta}
contains $e$.
\end{proof}

\section{Effect of the operation on the moduli space}
 \label{sc:polygon}

Suppose a dimer model $G'$ is obtained
from a consistent dimer model $G=(B, W, E)$ by 
the operation in Section \ref{sc:algorithm}.
Let $\Gamma$ and $\Gamma'$ be the quivers
associated with $G$ and $G'$ respectively, and
$S \subset E$ be the set of removed edges.
A vertex of $\Gamma'$ is the union of vertices of $\Gamma'$
connected by arrows in $S$.

Regarding quivers as categories,
we can define a functor $\phi : \Gamma \to \Gamma'$ as follows:
A vertex $v$ of $\Gamma$ is sent
to the vertex of $\Gamma'$ containing $v$.
An arrow $a$ is sent to itself if $a \notin S$,
and to the identity of the vertex containing $a$ if $a \in S$.
The functor $\phi$ induces the functor
$$
 \phi^* : \module \Gamma' \to \module \Gamma.
$$

For the stability parameter $\theta$ in Proposition \ref{prop:large_hexagon},
we define a stability parameter $\theta'$ for $\Gamma'$
such that $\theta'(v')$ is the sum of $\theta(v)$ for vertices $v \subset v'$.
Then the above functor gives an open embedding
$\scM'_{\theta'} \to \scM_{\theta}$.
In terms of the moduli spaces,
Proposition \ref{prop:polygon} is interpreted as follows:

\begin{proposition}\label{prop:complement}
The image of $\scM'_{\theta'}$ is the complement
of the toric divisor $D_\frakc \subset \scM_\theta$
corresponding to the removed corner $\frakc$.
\end{proposition}

As a corollary, we obtain Proposition \ref{prop:restriction}:
\begin{corollary}\label{cor:vanishing_only}
The edges in $S$
are exactly those which correspond
to morphisms between tautological bundles
vanishing only on the toric divisor
$D_\frakc \subset \scM_\theta$.
\end{corollary}
\begin{proof}
For an edge $e$, let $\Psi(e)$ denote the corresponding morphism
between tautological bundles on $\scM_\theta$.
First consider an edge $e \in S$.
By the construction of the corner perfect matching $D_\frakc$,
the morphism $\Psi(e)$ vanishes on $D_\frakc$.
On the other hand, Proposition \ref{prop:complement} shows
that $\Psi(e)$ does not vanish on any other toric divisor.
Next suppose $e \notin S$.
Then $e$ survives as an edge $e'$ in the consistent dimer model $G'$,
and $\Psi(e)$ restricts to a morphism $\Psi'(e')$
of tautological bundles on $\scM'_{\theta'}$.
By Corollary \ref{cor:corner},
there is a corner perfect matching $D'$ of $G'$ containing $e$.
Then the simplicity of $D'$ shows that
$\Psi'(e')$ vanishes along the divisor corresponding to $D'$.
\end{proof}

\section{Injectivity of the universal morphism}
 \label{sc:injectivity}

Let $G$ be a dimer model and
$\bigoplus_v \scL_v$ be the tautological bundle
on the moduli space $\scM_\theta$
of quiver representations
with respect to a generic stability parameter $\theta$.

\begin{proposition} \label{prop:inj} 
If $G$ is consistent,
then the universal morphism
$$
 \bC \Gamma \to \End \lb \bigoplus_v \scL_v \rb
$$
is injective.
\end{proposition}

\begin{proof}
A consistent dimer model is non-degenerate
by Proposition \ref{prop:non-degenerate}.
Therefore, the moduli space contains a three-dimensional
algebraic torus $\bT$ as an open set
by \cite[Proposition 5.1]{Ishii-Ueda_08}.
Fix a $\bT$-fixed point $[\Psi]$ on $\scM_\theta$
which is the isomorphism class of a representation $\Psi$ of $\Gamma$.
Then the toric affine open neighborhood $U_\Psi$ of $[\Psi]$ is isomorphic
to a closed subscheme of the space $\widetilde{\scM}$ of all the representations
of $\Gamma$ by \cite[Lemma 4.3]{Ishii-Ueda_08}.
Then $\bT \subset U_\Psi$ is lifted to a subgroup of the group $\widetilde{\bT}$
of $\bC^\times$-valued representations of $\Gamma$.
Thus $\bT$ acts on both $\bC\Gamma$ and $\End \lb \bigoplus_v \scL_v \rb$
in such a way that the homomorphims is equivariant.
If two paths $p$ and $q$ from $u$ to $v$ are not
equivalent, they are not weakly equivalent
by the first consistency condition,
and hence they have different weights
with respect to the $\bT$-action.
Since equivalence classes of paths form a basis of $\bC\Gamma$
and any path goes to a non-zero element in $\End \lb \bigoplus_v \scL_v \rb$,
the homomorphism is injective.
\end{proof}

\section{Preservation of the tilting condition:
$\ahilb(\bC^3)$ versus
$\ahilb(\bC^3) \setminus \ahilb(\bC^2)$}
\label{sc:tilting-G-Hilb}

Let $A$ be a finite small subgroup of $\GL_2(\bC)$
and set
$Y = \ahilb(\bC^2)$,
$U = \ahilb(\bC^3)$ and
$U' = U \setminus Y$.
Let $\scR_\rho$ be
the tautological bundle on
$U = \ahilb(\bC^3)$
corresponding to an irreducible representation $\rho$ of $A$,
and $\scR' _\rho= \scR_\rho|_{U'}$
be its restriction to $U'$.
In this section, we compare tilting conditions of $\bigoplus_\rho \scR_\rho$
and $\bigoplus_\rho \scR'_\rho$ and
prove two lemmas which will be used later in a more general setting.
%
%
We first prove a general result
that the restriction of a tilting object
to an open subset is also a generator:

\begin{lemma} \label{lm:restricting_generator}
Let $\scE$ be a tilting object in $D^b \coh U$.
Then the pull-back of $\scE$ by an open immersion $\iota : U' \to U$
is a generator in $D^b \coh U'$.
\end{lemma}

\begin{proof}
For any coherent sheaf $\scF$ on $U'$,
there is a coherent sheaf $\scFtilde$ on $U$
such that $\iota^* \scFtilde = \scF$,
see e.g. \cite[Exercise 5.15]{Hartshorne}.
Since $\scE$ is a tilting object,
$\scF$ is a direct summand of an object in $D^b \coh U$
obtained from $\scE$ by taking mapping cones.
Since derived restriction commutes with the operation of
taking mapping cones,
this shows that $\scF$ is obtained from $\iota^* \scE$
by taking direct summands and mapping cones.
This implies that $\iota^* \scE$ is a generator in $D^b \coh U'$.
\end{proof}

To compare tilting properties of $\bigoplus_{\rho} \scR_{\rho}$
and $\bigoplus_{\rho} \scR'_{\rho}$,
we use the exact sequence
\begin{equation}\label{eq:long_exact}
 \cdots \to H^{i}_Y(U, \scR_\rho^\vee \otimes \scR_\tau)
  \to H^i(U, \scR_\rho^\vee \otimes \scR_\tau)
  \to H^i(U', {\scR'_\rho}^\vee \otimes \scR'_\tau)
  \to \cdots.
\end{equation}
In this exact sequence, we have the following vanishing result.
\begin{lemma} \label{lm:vanishing_of_local_cohomology}
The local cohomology
$
 H^i_Y(U, \scR_{\rho}^{\vee} \otimes \scR_{\tau})
$
vanishes for $i \ge 2$. 
\end{lemma}

\begin{proof}
We use
$$
 H^i_Y(U, \scR_\rho^\vee \otimes \scR_\tau)
  \cong
   \varinjlim_{n}
    \Ext^i_{U}(\scO_{n Y}, \scR_\rho^\vee \otimes \scR_\tau)
$$
to compute the local cohomology.
One has
\begin{equation}\label{eq:ext^i}
\begin{aligned}
 \Ext^i_{U}(\scO_{n Y}, \scR_\rho^\vee \otimes \scR_\tau)
  &\cong
 \Ext^i_{U}(\{ \scO_U(- n Y) \to \scO_U \},
    \scR_\rho^\vee \otimes \scR_\tau) \\
  &\cong
 H^i(\{ \scO_U \to \scO_U(nY) \}
    \otimes \scR_\rho^\vee \otimes \scR_\tau) \\
  &\cong
 H^{i-1}( \scO_U(n Y)|_{n Y}
    \otimes \scR_\rho^\vee \otimes \scR_\tau).
\end{aligned}
\end{equation}
Since $U$ has the trivial canonical bundle,
the adjunction formula gives an isomorphism
$$
 \scO_U(nY)|_{nY} \cong \omega_{n Y}
$$
with the dualizing sheaf $\omega_{nY}$ of $n Y$.
Since $Y$ is a resolution of an affine surface,
one has
$
 H^2(\scE) = 0
$
for any coherent sheaf $\scE$ on $Y$.
It follows that any surjection
$\scF \to \scG \to 0$
of coherent sheaves on $Y$ induces
a surjection
$
 H^1(\scF) \to H^1(\scG) \to 0
$
of cohomology groups.
By definition of full sheaves,
one has
$
 H^1(\scR_{\rho}^{\vee} \otimes \omega_Y) = 0
$
and $\scR_{\tau}|_Y$ is generated by global sections.
The latter shows the existence of a surjection
$
 \scO_U^{\oplus N} \to \scR_{\tau}|_Y \to 0
$
for some $N \in \bN$,
which gives a surjection
$$
 \scR_\rho^\vee \otimes \omega_Y^{\oplus N}
  \to \scR_\rho^\vee \otimes \scR_\tau \otimes \omega_Y
  \to 0,
$$
which combined with
$
 H^1(\scR_{\rho}^{\vee} \otimes \omega_Y) = 0
$
gives
$$
 H^1(\scR_\rho^\vee \otimes \scR_\tau \otimes \omega_Y)
  = 0.
$$
This proves
$$
 \Ext^i_{U}(\scO_{n Y}, \scR_\rho^\vee \otimes \scR_\tau) = 0
$$
for $n=1$.

For $n>1$,
note the exact sequence
$$
 0 \to \scO_U(D-Y)|_{(n-1) Y}
  \to \scO_U(D)|_{n Y}
  \to \scO(D)|_Y
  \to 0
$$
which holds for any divisor $D$ on $Y$.
By substituting $D = n Y$,
one obtains
$$
 0 \to \scO_U((n-1)Y)|_{(n-1) Y}
  \to \scO_U(n Y)|_{n Y}
  \to \scO(n Y)|_Y
  \to 0,
$$
which is the same as
$$
 0
  \to \omega_{(n-1)Y}
  \to \omega_{nY}
  \to \omega_Y^{\otimes n}
  \to 0.
$$
Since $Y$ is the minimal resolution,
$\omega_Y^{\otimes n}$ is
generated by global sections and
one has
$$
 H^1(\scR_\rho^\vee \otimes \scR_\tau
  \otimes \omega_Y^{\otimes n}) = 0.
$$
by the same argument as above.
Together with the exact sequence
$$
 H^1(\scR_\rho^\vee \otimes \scR_\tau
  \otimes \omega_{(n-1) Y})
  \to
  H^1(\scR_\rho^\vee \otimes \scR_\tau \otimes
 \omega_{n Y})
  \to
H^1(\scR_\rho^\vee \otimes \scR_\tau
  \otimes \omega_Y^{\otimes n}),
$$
one can inductively show
$$
 H^1(\scR_\rho^\vee \otimes \scR_\tau
  \otimes \omega_Y^{\otimes n}) = 0
$$
for any positive integer $n$.
\end{proof}

We obtain the following corollary which we will not use.
Note that its assumption follows from \cite{Bridgeland-King-Reid}.

\begin{corollary} 
If the condition \textup{(\bfT)} holds for $U$,
then the direct sum $\bigoplus_{\rho} \scR'_{\rho}$
over the set of irreducible representations of $A$
is a tilting object.
\end{corollary}
\begin{proof}
The restriction $\bigoplus_{\rho} \scR'_{\rho}$ is a generator
by Lemma \ref{lm:restricting_generator}.
The vanishing of
$
 H^i({\scR'_\rho}^\vee \otimes \scR'_\tau)
$
for $i\ge 1$ follows from the long exact sequence \eqref{eq:long_exact}
and Lemma \ref{lm:vanishing_of_local_cohomology}.
\end{proof}

\section{Preservation of surjectivity:
$\ahilb(\bC^3)$ versus \\
$\ahilb(\bC^3) \setminus \ahilb(\bC^2)$}
 \label{sc:surj-G-Hilb}

We use the same notation
as in Section \ref{sc:tilting-G-Hilb}.
Let $\Lambda$ be the McKay quiver of $A$,
and $\Lambda'$ be the quiver
obtained from $\Lambda$ by adding inverse arrows
to the arrows starting from special representations
corresponding to ``multiplication by $z$''.

We prove the following in this section:

\begin{proposition} \label{prop:surjective}
The natural map from $\bC \Lambda'$
to the endomorphism algebra
of $\bigoplus_i \scR_i'$ is surjective.
\end{proposition}

Let $\Ntilde = \bZ^3$ be the group of one-parameter subgroups
of the dense torus in $\bC^3$.
The group $N \supset \Ntilde$
of one-parameter subgroups
of the dense torus in $U = \ahilb(\bC^3)$ is given by
$$
 N = \bZ^3 + \bZ\cdot\frac{1}{n}(1, q, n-(1+q)),
$$
and
the fan describing the quotient $\bC^3/A$
has the unique 3-dimensional cone
given by the first quadrant
$
 (\bR_{\ge 0})^3
  \subset \Ntilde_\bR = N_\bR
$.

\begin{lemma}[Craw and Reid \cite{Craw-Reid}]
One-dimensional cones in the fan
describing $U = \ahilb(\bC^3)$
which are adjacent to $\bR_{\ge 0} (0,0,1)$
are generated by
$$
 \frac{1}{n}(j_t, i_t, n-(i_t+j_t)) \in N
$$
for $0 \le t \le r+1$.
Here we say two one-dimensional cones are adjacent if
they are contained in a common two-dimensional cone.
\end{lemma}


Now let us express the tautological bundles
as $\bQ$-linear combinations of exceptional divisors,
i.e., toric divisors except the three which correspond
to the corners of the junior simplex.
Let $x, y, z \in \bC[\Mtilde]$
be the coordinates of
$
 \bC^3 = \Spec \bC[x, y, z]
$
corresponding to the standard basis
of $\Mtilde = \Hom(\Ntilde, \bZ) \cong \bZ^3$.
Then rational sections of $\scR_d$ form a vector space with a
basis consisting of Laurent monomials
$x^ay^bz^c$ with $a+bq-(1+q)c \equiv d \mod n$.
On the other hand,
the coordinate ring of the dense torus in $\bC^3/A$
is given by
$
 \bC[x^{\pm 1}, y^{\pm 1}, z^{\pm 1}]^A
  = \bC[M],
$
where
\begin{align*}
 M &= \Hom(N, \bZ) \\
  &= \lc (a, b, c) \in \Mtilde \relmid a + b q - (1 + q) c \equiv 0 \mod n \rc.
\end{align*}
It follows that one can embed the line bundle $\scR_d^{\otimes n}$
into $\scO_{U}$ in a natural way
and it defines an effective exceptional divisor
$E_d$ on $U$ with $\scR_d^{\otimes n} = \scO_{U}(-E_d)$.

Let $C=(c_{st})_{s, t=1}^r$ be the negative of the intersection matrix of the
resolution $Y \to \bC^2 / A$;
$$
 c_{st}
  = \begin{cases}
      b_s & s = t, \\
      -1  & |s-t| = 1, \\
       0  & \text{otherwise}.
\end{cases}
$$
The lower-right principal minors
$$
 i_t =
\begin{vmatrix}
 b_{t+1} & -1 \\
 -1 & b_{t+2} & -1 \\
  & -1 & \ddots & \ddots \\
  & & \ddots & b_{r-1} & -1 \\
  & & & -1 & b_r
\end{vmatrix}
$$
give the integers
appearing in the continued fraction expansion
in Section \ref{sc:continued-fraction},
since they satisfy \eqref{eq:cont-frac}.
In particular,
one has $\det C = i_0 = n$.
Let $\eta_{st}$ be the $(s, t)$-th entry of the integer matrix $n C^{-1}$.
Since
\begin{align*}
\begin{pmatrix}
  b_1 & -1 \\
 -1 & b_2 & -1 \\
  & -1 & \ddots & \ddots \\
  & & \ddots & \ddots & -1 \\
  & & & -1 & b_r
\end{pmatrix}
\begin{pmatrix}
 i_1 \\
 i_2 \\
 i_3 \\
 \vdots \\
 i_r
\end{pmatrix}
 =
\begin{pmatrix}
 b_1 i_1 - i_2 \\
 - i_1 + b_2 i_2 - i_3 \\
 - i_2 + b_3 i_3 - i_4 \\
 \vdots \\
 - i_{r-1} + b_r i_r
\end{pmatrix}
 =
\begin{pmatrix}
 i_0 \\
 0 \\
 0 \\
 \vdots \\
 0
\end{pmatrix}
 =
\begin{pmatrix}
 n \\
 0 \\
 0 \\
 \vdots \\
 0
\end{pmatrix}
\end{align*}
and
\begin{align*}
\begin{pmatrix}
  b_1 & -1 \\
 -1 & b_2 & -1 \\
  & \ddots & \ddots & \ddots \\
  & & -1 & b_{r-1} & -1 \\
  & & & -1 & b_r
\end{pmatrix}
\begin{pmatrix}
 j_1 \\
 j_2 \\
 \vdots \\
 j_{r-1} \\
 j_r
\end{pmatrix}
 =
\begin{pmatrix}
 b_1 j_1 - j_2 \\
 - j_1 + b_2 j_2 - j_3 \\
 \vdots \\
 - j_{r-2} + b_{r-1} j_{r-1} - j_r \\
 - j_{r-1} + b_r j_r
\end{pmatrix}
 =
\begin{pmatrix}
 0 \\
 0 \\
 \vdots \\
 0 \\
 j_{r+1}
\end{pmatrix}
 =
\begin{pmatrix}
 0 \\
 0 \\
 \vdots \\
 0 \\
 n
\end{pmatrix},
\end{align*}
one has
\begin{equation}\label{equation:1retsu}
i_t = \eta_{t1} \text{ and }j_t = \eta_{tr}
\end{equation}
for $1 \le t \le r$.

Let $D_t$ be the divisor on $U$ corresponding to the ray
$\bR_{\ge 0}(j_t, i_t, n-(i_t+j_t))$ in $N_{\bR}$.
Since a line bundle on $Y$ is determined
by the degrees of the restrictions to the exceptional curves,
the fact that
\begin{align*}
 \deg \scO(-E_{i_s})|_{Y\cap D_t}
  &= \deg \scR_{i_s}^{\otimes n}|_{Y \cap D_t} \\
  &= n \deg \scR_{i_s}|_{Y \cap D_t} \\
  &= n \deg \scM_{i_s}|_{C_t} \\
  &= n \delta_{st}
\end{align*}
implies the following:

\begin{lemma}
We can write
$$
 E_{i_s} = \sum_{t=1}^r \eta_{st}D_t
  + (\text{sum of other exceptional divisors}).
$$
Therefore, for an integer $d=\sum_t d_t i_t$
as in Theorem \ref{theorem:wunram},
the coefficient of $D_t$ in $E_d$ is
$\sum_s d_s \eta_{st}$.
\end{lemma}

For integers $f, g \in [0, n-1]$, 
the space of rational sections of $\scR_f^{\vee} \otimes \scR_g$ has
$$
 \lc x^ay^bz^c \relmid a + bq -c(1+q) \equiv g-f \mod n \rc
$$
as a basis.
Write $f=\sum_t f_t i_t$ and $g=\sum_t g_t i_t$
as in Theorem \ref{theorem:wunram}.

\begin{corollary} \label{cr:zero}
For integers $a, b, c$ with $a + bq -c(1+q) \equiv g-f \mod n$,
the order of zero of the rational section $x^ay^bz^c$
of $\scR_f^{\vee} \otimes \scR_g$
along $D_t$ is given by the integer
\begin{align} \label{eq:et}
 e_t
  := \frac{1}{n} \lb a j_t + b i_t + c(n-(i_t+j_t))
  - \sum_{s=1}^r (g_s - f_s) \eta_{st} \rb.
\end{align}
\end{corollary}

Indeed,
the order of zero of $x^a y^b z^c$ along $D_t$
as a section of $\scO_U$ is given by
$
 a j_t + b i_t + c(n-(i_t + j_t)),
$
and the difference between the order of zero
as a section of $\scO_U(-E_d) \cong \scR_d^{\otimes n}$
and that of $\scO_U$ is given by
$\sum_s (g_s - f_s) \eta_{st}$.
It follows from \pref{cr:zero} that 
a rational section $x^a y^b z^c$
of $\scR_f^\vee \otimes \scR_g$
is holomorphic on $U'$
only if 
\begin{equation} \label{eq:inequality}
 a \ge 0, \quad 
 b \ge 0, \quad \text{and} \quad
 e_t \ge 0 \quad (1 \le t \le r).
\end{equation}
By substituting $t = 1$ in \eqref{eq:et},
one obtains
\begin{align*}
 e_1
  &= \frac{1}{n} \lb a j_1 + b i_1 + c(n-(i_1+j_1))
  - \sum_{s=1}^r (g_s - f_s) \eta_{s1} \rb \\
  &= \frac{1}{n} \lb a + b q + c(n-(q+1))
  - \sum_{s=1}^r (g_s - f_s) i_s \rb \\
  &= \frac{1}{n} \lb a+bq+c(n-1-q)-(g-f) \rb,
\end{align*}
and the condition
$a+bq-c(1+q) \equiv g-f \mod n$ is satisfied
if $e_1$ is an integer.

By multiplying the matrix $C$ to \eqref{eq:et},
one obtains
\begin{align*}
 \sum_{t=1}^r c_{st} e_t
  &= \frac{1}{n} \sum_{t=1}^r c_{st} \lb a j_t + b i_t + c(n-(i_t+j_t))
  - \sum_{u=1}^r (g_u - f_u) \eta_{ut} \rb \\
  &= \frac{1}{n} \sum_{t=1}^r c_{st} \lb a \eta_{tr} + b \eta_{t1}
   + c(n-(\eta_{t1}+\eta_{tr}))
  - \sum_{u=1}^r (g_u - f_u) \eta_{ut} \rb \\
  &= a \delta_{sr} + b \delta_{s1}
   + c \lb \sum_{t=1}^r c_{st} - \delta_{s1}-\delta_{sr} \rb
   - (g_s - f_s),
\end{align*}
which gives
\begin{equation}\label{equation:system}
\left\{
\begin{aligned}
b - b_1 e_1 + e_2 &= g_1 - f_1 - (b_1-2)c, \\
e_{t-1} - b_t e_t + e_{t+1} &= g_t - f_t -(b_t-2)c, & 2 \le t \le r-1, \\
e_{r-1} - b_r e_r + a &= g_r - f_r - (b_r-2)c.
\end{aligned}
\right.
\end{equation}
If $x^ay^bz^c$ is a holomorphic section
of $\scR_f^{\vee} \otimes \scR_g$ on $U'$,
then the solution $(e_t) \in \bZ^r$ to \eqref{equation:system}
must satisfy \eqref{eq:inequality}.
Putting $e_0:=b$ and $e_{r+1}:=a$,
we consider the second difference
\begin{equation*}
e_t'':=e_{t-1} - 2 e_t + e_{t+1}
\end{equation*}
for $1 \le t \le r$.
Then \eqref{equation:system} can be written as
\begin{equation}\label{equation:difference}
e_t''= g_t-f_t +(b_t-2)(e_t-c) \qquad (1 \le t \le r).
\end{equation}
If $e_t'' \ge 0$ for all $t$, then the function $t \mapsto e_t$ is convex.
This is not true in general but the situation is very close as we will see now.
To estimate $e_t''$ from below, we use the following lemma:

\begin{lemma}\label{lemma:integers}
Let $e \ge 0$, $b_t \ge 2$, $f_t \le b_t-1$ and $c<0$ be integers. Then
\begin{enumerate}
\item $-f_t +(b_t-2)(e-c) \ge -1$. 
\item If $-f_t +(b_t-2)(e-c) = -1$, then $f_t=b_t-1$.
\item If $-f_t +(b_t-2)(e-c) = 0$, then $f_t \ge b_t -2$.
\end{enumerate}
\end{lemma}

We omit the proof, which is elementary and straightforward.
Since $(f_1, \dots, f_r)$ satisfies the condition
in Lemma \ref{lemma:wunramvanishing},
this implies the following:

\begin{corollary}\label{corollary:difference}
Suppose $(e_t)_{t=0}^{r+1} \in \bZ^{r+2}$ is an integer solution
to the difference equation
\eqref{equation:difference}
for $c<0$, and
$f=\sum_t f_t i_t$, $g=\sum_t g_t i_t$
as in Theorem \ref{theorem:wunram}.
Then we have the following:
\begin{enumerate}
\item For a fixed $t$, $e_t \ge 0$ implies $e''_t \ge -1$.
\item If $e''_s=e''_t=-1$ for $s<t$ and $e_u \ge 0$ for any $u \in [s, t]$, then there is $l \in (s, t)$ with $e''_l \ge 1$.
\item If $e_{\alpha-1} > e_{\alpha} \ge 0$ for some $\alpha \ge 1$,
then we have $e_0 \ge \dots \ge e_{\alpha-1} > e_{\alpha}$
\item If $0 \le e_{\beta} < e_{\beta+1}$ for some $\beta \le r$,
then we have $e_{\beta} < e_{\beta+1} \le \dots \le e_{r+1}$
\end{enumerate}
\end{corollary}

In particular,
if $e_t \ge 0$ for all $t$,
then there are integers $p$ and $p'$
with $0 \le p \le p' \le r+1$ such that
\begin{equation}\label{equation:downup}
e_0 \ge \dots  \ge e_{p-1}> e_{p} = \dots = e_{p'} < e_{p' +1} \le
\dots \le e_{r+1}
\end{equation}

The following is the key to the proof of
Proposition \ref{prop:surjective}:

\begin{lemma} \label{lm:inequality}
Let $x^a y^b z^c$ be a rational section
of $\scR_f^{\vee} \otimes \scR_g$ satisfying
\eqref{eq:inequality}.
If $c$ is negative,
then there exist a special representation $i_s$ and
a rational section $x^{a'}y^{b'}z^c$ of
$\scR_f^{\vee} \otimes \scR_{i_s}$
satisfying $0 \le a' \le a$, $0 \le b' \le b$, and
$$
 h_t := \frac{1}{n}
  \left(
   a' j_t + b' i_t + c (n - (i_t + j_t))
    - \sum_u (\delta_{us} - f_u) \eta_{ut}
  \right)
   \ge 0,
 \qquad 1 \le t \le r.
$$
\end{lemma}

\begin{proof}
Since the claim is obvious if $g$ is special,
we assume that $g$ is not special.
First note that it suffices to show that for a suitable choice of $s$,
there is a solution
$
 (h_0, \dots, h_{r+1}) \in (\bZ_{\ge 0})^{r+2}
$
to
\begin{equation} \label{equation:difference2}
 h''_t = \delta_{ts}-f_t + (b_t-2)(h_t-c),
  \qquad 1 \le t \le r,
\end{equation}
with $0 \le h_t \le e_t$
for $0 \le t \le r+1$.
Indeed, if $(h_t)$ is such a solution,
then $a':=h_{r+1}$ and $b':=h_0$
determine a desired rational section
$x^{a'}y^{b'}z^c$ of $\scR_f^{\vee} \otimes \scR_{i_s}$.
Note also that an integer solution $(h_t) \in \bZ^{r+2}$ satisfying \eqref{equation:difference2}
(without the assumption $h_t \ge 0$) is
determined by any two consecutive values
$h_{\alpha}, h_{\alpha+1}$.
Thus all we have to do is to choose suitable $s$ and values $h_{\alpha}, h_{\alpha+1}$ for some $\alpha$
such that the corresponding solution $(h_t) \in \bZ^{r+2}$ to \eqref{equation:difference2} satisfies $0
\le h_t \le e_t$.

Let $0 \le p \le p' \le r+1$ be as in \eqref{equation:downup}
and put
$$e:=e_p(=e_{p'}),$$
which is the minimum value of $e_t$.
We note that if $p<t<p'$, then $e_t''=0$ and
\begin{equation}\label{equation:non-positive}
-f_t+(b_t-2)(e-c)=e''_t-g_t=-g_t \le 0.
\end{equation}
Let $q$ be the integer determined by
$$
q:=\max \left\{t \in \bZ \mid 1 \le t \le p \text{ and }-f_t + (b_t-2)(e-c) >0\right\}
$$
if this set is non-empty, and put $q=0$ otherwise.
Similarly, let $q'$ be the integer determined by
$$
q':=\min \left\{ t \in \bZ \mid p' \le t \le r \text{ and }-f_{t} + (b_{t}-2)(e-c) >0\right\}
$$
if this set is non-empty, and put $q'=r+1$ otherwise.
Since we have \eqref{equation:non-positive} for $t \in (p, p')$, our choice of $q$ and $q'$ implies
\begin{equation}\label{equation:non-positive2}
 -f_t+(b_t-2)(e-c) \le 0,
  \qquad q < t < q'.
\end{equation}
We first consider the case where there is an integer $v \in (q, q')$ such that
$$-f_v + (b_v-2)(e-c) <0.$$
In this case, we have $f_v= b_v-1$ and
$-f_v + (b_v-2)(e-c)=-1$ by Lemma \ref{lemma:integers}.
Such an integer $v \in (q, q')$ is unique by \eqref{equation:non-positive2}, Lemma
\ref{lemma:wunramvanishing} and Lemma \ref{lemma:integers}.
Thus if $t \in (q, q')$ and $t \ne v$, then
\begin{equation}\label{eq:straight}
-f_t + (b_t-2)(e-c)=0.
\end{equation}
Now we choose $s$ as follows.
\begin{enumerate}
\item[(1)] If $v \in [p, p']$, then $s:=v$.
\item[(2)] If $v < p$, then $s:=p$.
\item[(3)] If $v > p'$, then $s:=p'$.
\end{enumerate}
Note that $e_s=e$  and $q < s < q'$ in all cases.
We have $e''_s \ge 0 > -f_s + (b_s-2)(e_s-2)$ in (1) and
$e''_s > 0 = -f_s + (b_s-2)(e_s-2)$ in (2) and (3).
Thus $e''_s  >-f_s + (b_s-2)(e_s-2)$ holds in all cases
and we obtain $g_s>0$.
This means that
$$
\delta_{st} \le g_t
$$
holds for any $t$.

Now we define $(h_t)$ satisfying \eqref{equation:difference2} by the
following two consecutive values:
\begin{enumerate}
\item[(1)] If $v \in [p, p']$, then $h_p = h_{p+1}= e$.
\item[(2)] If $v < p$, then $h_{p}=h_{p+1}=e$.
\item[(3)] If $v > p'$, then $h_{p'-1}=h_{p'}=e$.
\end{enumerate}
Then, by \eqref{eq:straight} and by our choice of $q$, $q'$ and $s$, it satisfies
\begin{enumerate}
\item[(1)]
$h_{q-1}>h_q= \dots = h_{q'} < h_{q'+1}.$
\item[(2)]
$h_{p-1}>h_p= \dots = h_{q'} < h_{q'+1}.$
\item[(3)]
$h_{q-1}>h_q= \dots = h_{p'} < h_{p'+1}.$
\end{enumerate}
in each case.
By Corollary \ref{corollary:difference},
we see that $h_t \ge e \ge 0$ for any $t$.
To compare $h_t$ and $e_t$, note that ($h_p=e_p$ and $h_{p+1}\le e_{p+1}$)
or ($h_{p'-1}\le e_{p'-1}$ and $h_{p'}=e_{p'}$) hold.
Moreover, by our choice of $s$, we have $\delta_{st} \le g_t$ for any $t$.
Therefore, we inductively obtain $h''_t \le e''_t$ and $h_t \le e_t$.

The case where there is no such $v$ is similar and easier.
If $q \ne q'$, we can take any $s$ with $g_s >0$
and we can define $(h_t)$ by $h_q=h_{q+1}=e$.
When $q=q'$, we have $e''_q=g_q-f_q + (b_q-2)(e_q-2)\ge 2$.
If $-f_q + (b_q-2)(e_q-2)=1$, then since $g_q>0$, we can
take $s=q$ and we can define $(h_t)$ by $h_q=e$, $h_{q+1}=e+1$.
If $-f_q + (b_q-2)(e_q-2)\ge 2$, then
take any $s$ with $g_s >0$ and define $(h_t)$ by $h_q=e$, $h_{q+1}=e+1$.
\end{proof}

Now we prove Proposition \ref{prop:surjective}:

\begin{proof}[Proof of Proposition \ref{prop:surjective}]
Recall that a path in $\Lambda'$ is obtained
by concatenating paths in $\Lambda$ and
inverse arrows to the arrows in $\Lambda$
corresponding to ``multiplication by $z$''
from special representations.
We show that if $x^ay^bz^c$ is a rational section of
$\scR_f^\vee \otimes \scR_g$
satisfying \eqref{eq:inequality},
then there is a path in $\Lambda'$ from $f$ to $g$
that is mapped to $x^a y^b z^c$.
Since the assertion is obvious if $c$ is non-negative,
we assume that $c$ is negative.
Then, we have $s$, $a'$ and $b'$
as in Lemma \ref{lm:inequality}.
We can regard $x^{a'}y^{b'}z^{c+1}$ as a rational map from $\scR_f$ to
$\scR_{i_s + n-q-1}$,
whose orders of zeros along the divisors $D_t$ are the same as those of
$x^{a'}y^{b'}z^c$ by Corollary \ref{cor:vanishing_only}.
Therefore, we can represent the rational map $x^ay^bz^c: \scR_f \to \scR_g$ as the product of
the rational maps $ x^{a'}y^{b'}z^{c+1}:\scR_f \to \scR_{i_s + n - q - 1}$,
$z^{-1}: \scR_{i_s + n - q - 1} \to \scR_{i_s}$,
and $x^{a-a'}y^{b-b'}:\scR_{i_s} \to \scR_g$.
The last rational map corresponds to a path in the McKay quiver
and we can prove the assertion by induction on $-c$.
\end{proof}

The proof of Proposition \ref{prop:surjective} also shows the following:

\begin{corollary} \label{cr:reflexivity}
A rational section $x^a y^b z^c$
of $\scR_f^\vee \otimes \scR_g$
is holomorphic on $U'$
if and only if \eqref{eq:inequality} is satisfied.
\end{corollary}

%

\section{Some technical lemmas}

This section is devoted to the proof of technical lemmas
on the paths of the quiver associated with a dimer model,
which will be needed later.
%
Consider a pair of zigzag paths with adjacent slopes,
which give a corner perfect matching $D$
as in Section \ref{sc:large-hexagon}.
We have a functor
$$
\phi_\frakc:\Gamma \to \Lambda
$$
with respect to the corner $\frakc$ corresponding to $D$
as in Section \ref{sc:corner},
where $\Lambda$ is the McKay quiver
whose vertices are large hexagons.
There is a corner perfect matching $\bar{D}$ of $\Lambda$
corresponding to $D$,
which consists of the arrows representing
``multiplications by $z$''.

\begin{lemma}\label{lemma:lifting}
Let $v$ be a vertex of $\Gamma$.
\begin{enumerate}
 \item
Suppose $v$ is the source of the large hexagon $\phi_\frakc(v)$ and
a path $p$ of $\Lambda$ starting from $\phi_\frakc(v)$
does not intersect with $\bar{D}$.
Then there is a path $\tilde{p}$ of $\Gamma$
from $v$ to any vertex in the large hexagon $t(p)$
such that $\phi_\frakc(\tilde{p}) = p$ and
$\tilde{p}$ does not intersect with $D$.
 \item
Suppose $v$ is the sink of the large hexagon $\phi_\frakc(v)$ and
a path $p$ of $\Lambda$ ending at $\phi_\frakc(v)$
does not intersect with $\bar{D}$.
Then there is a path $\tilde{p}$ of $\Gamma$
from any vertex in the large hexagon $s(p)$ to $v$
such that $\phi_\frakc(\tilde{p}) = p$ and
$\tilde{p}$ does not intersect with $D$.
\end{enumerate}
\end{lemma}

The first assertion follows from the following lemma.
We can also show the dual statement, which implies the second assertion above.

\begin{lemma}
Suppose a vertex $v$ of $\Gamma$ is the source of the large hexagon $\phi_\frakc(v)$.
\begin{enumerate}
\item 
For any vertex $w$ of $\Gamma$ in $\phi_\frakc(v)$,
there is a path $q$ from $v$ to $w$ with $\phi_\frakc(q)=e_{\phi_\frakc(v)}$
(the idempotent of $\phi_\frakc(v)$)
which doesn't contain arrows in $D$.
\item
If $a$ is an arrow of $\Lambda$ with $s(a)=\phi_\frakc(v)$ and $a \notin \bar{D}$,
then there is a path $q'$ from $v$ to the source $w$ of the large hexagon $t(a)$
with $\phi_\frakc(q')=a$ which doesn't contain arrows in $D$.
\end{enumerate}
\end{lemma}
\begin{proof}
For the first assertion,
let $w$ be a vertex in $\phi_\frakc(v)$ and
take the minimal path $q$ from $v$ to $w$ inside $\phi_\frakc(v)$.
Then, by the construction of the corner perfect matching $D$,
$q$ doesn't contain arrows in $D$.

For the second assertion,
one of the two zigzag paths used to construct the large hexagons
contacts both $v$ and $w$,
and one can take the path from $v$ to $w$ on $\Gamma$
parallel to this zigzag path
as $q'$.
\end{proof}

\begin{lemma}\label{lemma:go_to_source}
Suppose $a$ is an arrow of $\Gamma$ contained in the perfect matching $D$.
Then there is a path $q$ of $\Gamma$ with the following properties:
\begin{itemize}
\item
$q$ goes from $s(a)$ to the source $w$ of the large hexagon that is adjacent to the sink $u$ of $\phi_\frakc(t(a))$ by the arrow $b$ in $D$ with $s(b)=w$ and $t(b)=u$.
\item
$\phi_\frakc(b q)$ is equivalent to $\phi_\frakc(a)$.
\item
$q$ doesn't contain arrows in $D$.
\end{itemize}
\end{lemma}
\begin{proof}
Recall from Section \ref{sc:small_minimal_weak}
that two paths are equivalent if and only if they
have the same homology class and they contain the same number
of arrows in $D$.
First assume that $a$ is inside a large hexagon
(i.e., $\phi_\frakc(s(a))=\phi_\frakc(t(a))$)
as in Figure \ref{fg:large-square-tile-zigzag-pm1}.
Then there is a minimal path $q'$ from $s(a)$ to $u$ inside $\phi_\frakc(t(a))$.
In this case,
$q$ is obtained by composing $q'$ and the path from $u$ to $w$
that goes around a node.
Next consider the case where $a$ is on one of the two zigzag paths
determining large hexagons but not on the other one
as in Figure \ref{fg:large-square-tile-zigzag-pm2}.
In this case,
$q$ is the path parallel to the zigzag path on which $a$ is lying.
Finally, suppose that $a$ is on the intersection of the two zigzag paths
as in Figure \ref{fg:large-square-tile-zigzag-pm3}.
In this case, $b$ coincides with $a$ and we can put $q=e_{s(a)}$.

\begin{figure}[htbp]
\centering
\input{large_square_tile_zigzag_pm2.pst}
\caption{Case 1}
\label{fg:large-square-tile-zigzag-pm1}
\end{figure}

\begin{figure}[htbp]
\centering
\input{large_square_tile_zigzag_pm3.pst}
\caption{Case 2}
\label{fg:large-square-tile-zigzag-pm2}
\end{figure}

\begin{figure}[htbp]
\centering
\input{large_square_tile_zigzag_pm4.pst}
\caption{Case 3}
\label{fg:large-square-tile-zigzag-pm3}
\end{figure}

\end{proof}

Lemma \ref{lemma:lifting}, \ref{lemma:go_to_source} and its dual yield the following:

\begin{lemma}\label{lemma:cancel}
Let $a$ be an arrow of $\Gamma$ contained in the perfect matching $D$.
\begin{itemize}
\item
Suppose $p$ is a path from $t(a)$ to the sink $u$ of some large hexagon and $p$ does not contain arrows in $D$.
Let $b$ be the arrow such that $t(b)=u$ and $s(b)$ is the source of the adjacent large hexagon. 
Then, there is a path $p'$ such that $pa$ is equivalent to $bp'$.
\item
Suppose $q$ is a path from the source $u$ of some large hexagon to $s(a)$ and $q$ does not contain arrows in $D$.
Let $c$ be the arrow such that $s(c)=u$ and $t(c)$ is the sink of the adjacent large hexagon. 
Then, there is a path $q'$ such that $a q$ is equivalent to  $q' c$.
\end{itemize}
\end{lemma}

\section{Preservation of the tilting condition: the general case}
 \label{sc:tilting-general}

Let $\Gamma$ be the quiver with relations
associated with a consistent dimer model,
and $\Gamma'$ be another quiver
obtained from $\Gamma$
by adding inverse to the arrows
from the sources of special large hexagons
to the sinks of the neighboring large hexagons
corresponding to ``multiplication by $z$''.
Let $\scM$ be the moduli space of representations of $\Gamma$
with the stability parameter chosen in Section \ref{sc:corner},
so that $\scM$ contains $U=\ahilb(\bC^3)$ as an open subscheme and
$Y = \ahilb(\bC^2)$ as a closed subscheme
for some finite abelian small subgroup $A$ of $\GL_2(\bC)$.
The McKay quiver of $A$ as a subgroup of $\SL_3(\bC)$
is denoted by $\Lambda$.
The moduli space $\scM$ carries the tautological bundles $\scL_v$
corresponding to vertices $v$ of $\Gamma$.
Let $\scM'$ be the complement $\scM \setminus Y$
and $\scL'_{v}$ be the restriction of $\scL_v$ to $\scM'$.
The restrictions of $\scL_v$ and $\scL'_v$
to $U$ and $U' = U \setminus Y$
give the tautological bundle $\scR_{\phi_\frakc(v)}$
on $U = \ahilb(\bC^3)$ and its restriction $\scR_{\phi_\frakc(v)}'$
to $U' = \ahilb(\bC^3) \setminus \ahilb(\bC^2)$ respectively.
We prove the following in this section:

\begin{proposition} \label{prop:tilting-general}
$\bigoplus_{v \in V} \scL_v$ is a tilting object
if and only if so is $\bigoplus_{v \in V} \scL'_{v}$.
\end{proposition}

\begin{proof}
In both directions, we use the long exact sequence
\begin{equation} \label{equation:long-exact}
 \cdots \to H^{i}_Y(\scM, \scL_v^{\vee}\otimes \scL_w)
  \to H^i(\scM, \scL_v^{\vee}\otimes \scL_w)
  \to H^i(\scM', {\scL'_{v}}^{\vee} \otimes \scL'_{w})
  \to \cdots.
\end{equation}
Since $Y$ is contained in $U$,
one has
$
 H^{i}_Y(\scM, \scL_v^{\vee}\otimes \scL_w)
  \cong H^{i}_Y(U, \scR_{\phi_\frakc(v)}^{\vee}\otimes \scR_{\phi_\frakc(w)})
$
and the ``only if'' part  follows immediately
from
Lemma \ref{lm:restricting_generator} and
Lemma \ref{lm:vanishing_of_local_cohomology}.

To show the ``if'' part,
assume that $\bigoplus_{v} \scL_{v}'$ is a tilting object.
In this case,
\pref{lm:vanishing_of_local_cohomology} and \eqref{equation:long-exact} implies
the vanishing of $H^i(\scM, \scL_v^{\vee}\otimes \scL_w)$
for $i \ge 2$,
and for acyclicity it suffices to show the surjectivity of 
\begin{equation}\label{equation:restriction}
 H^0(\scM', {\scL'_{v}}^{\vee} \otimes \scL'_{w})
  \to H^1_Y(\scM, \scL_v^{\vee}\otimes \scL_w).
\end{equation}
Put $\scL_{vw}:=\scL_v^{\vee} \otimes \scL_w$ and note that
$$
 H^1_Y(\scM, \scL_{vw})
  \cong H^1_Y(U, \scL_{vw}|_U)
  \cong \varinjlim_l \Ext^1_{\scO_U}(\scO_{lY}, \scL_{vw}|_U)
  \cong \varinjlim_l H^0(\scL_{vw} \otimes \scO_{lY}(lY)),
$$
where the last isomorphism follows from \eqref{eq:ext^i}.
Then the surjectivity of \eqref{equation:restriction} follows from the surjectivity of
$$
H^0(\scL_{vw}(lY)) \to H^0(\scL_{vw} \otimes \scO_{lY}(lY))
$$
for each $l>0$,
which is reduced to the surjectivity of
$$
H^0(\scL_{vw}(lY)) \to H^0(\scL_{vw}(lY)|_Y)
$$
by induction on $l$ with the aid of the commutative diagram
$$
\begin{CD}
0 @>>> H^0(\scL_{vw}((l-1)Y)) @>>> H^0(\scL_{vw}(lY)) @>>> H^0(\scL_{vw}(lY)|_Y) \\
@.     @VVV                        @VVV                        @|  \\
0 @>>> H^0(\scL_{vw}\otimes \scO_{(l-1)Y}((l-1)Y)) @>>> H^0(\scL_{vw}\otimes \scO_{lY}(lY)) @>>> H^0(\scL_{vw}(lY)|_Y).
\end{CD}
$$
Now, for a fixed $l$, $H^0(\scL_{vw}(lY)|_Y)$ has a basis of the form $x^ay^bz^{-l}$
satisfying \eqref{eq:inequality} where we replace $c$ with $-l$.
Then \pref{cr:reflexivity} shows that it can be lifted to a section of $\scL_{vw}|_{U'}$
and therefore is given by a path of $\Lambda'$ by
Proposition \ref{prop:surjective}.
Moreover,
by the proof of Proposition \ref{prop:surjective} and the assumption $l>0$,
the path can be chosen so that it contains an inverse arrow
(corresponding to ``multiplication by $z^{-1}$'' to a special representation)
but not arrows in the corner perfect matching $\bar D$.
Since
\begin{itemize}
\item
an inverse arrow in $\Lambda'$ can
be lifted to an inverse arrow arrow of $\Gamma'$ going from a sink to a source,
\item a path to the source of an inverse arrow in $\Lambda'$ can be lifted to
a path from an arbitrary vertex in the large hexagon to the source of the corresponding inverse arrow in $\Gamma'$ by 
the second statement of Lemma \ref{lemma:lifting} and
\item a path from the target of an inverse arrow in $\Lambda'$ can be lifted to
a path to an arbitrary vertex in the large hexagon from the source of the corresponding inverse arrow in $\Gamma'$ by 
the first statement of Lemma \ref{lemma:lifting},
\end{itemize}
the path can be lifted to a path of $\Gamma'$ from $v$ to $w$ and
\eqref{equation:restriction} is surjective.

Finally, we show that $\bigoplus_v \scL_v$ is a generator.
For an object $\alpha$ of $D^b \coh \scM$,
assume that $\RHom(\bigoplus_v \scL_v, \alpha)=0$.
Let $s$ be the source of the large hexagon
corresponding to a special representation of $A$ and
$t$ be the sink of the adjacent large hexagon which is the target of ``multiplication by $z$''
from the special representation.
Let $\iota$ denote the closed immersion $Y \to \scM$.
\pref{lm:restriction} below shows that
$$
 \iota_* \iota^* \scL_s^\vee \cong \{ \scL_t^\vee \to \scL_s^\vee \},
$$
so that one has
\begin{align*}
 \RHom(\iota^* \scL_s, \iota^* \alpha)
  &= \RGamma ( (\iota^* \scL_s)^\vee \otimes \iota^* \alpha) \\
  &= \RGamma ( \iota^* \scL_s^\vee \otimes \iota^* \alpha) \\
  &= \RGamma ( \iota_*( \iota^* \scL_s^\vee \otimes \iota^* \alpha)) \\
  &= \RGamma ( \iota_* \iota^* \scL_s^\vee \otimes \alpha)) \\
  &= \RGamma ( \{ \scL_t^\vee \to \scL_s^\vee \} \otimes \alpha)) \\
  &= 0.
\end{align*}
Since $\bigoplus \iota^* \scL_s$ is a tilting object on $Y$
by Theorem \ref{th:VdB},
we have $\iota^* \alpha =0$.
It follows that $\Supp \alpha \subset \scM'$ and
we obtain $\alpha =0$ by our assumption that
$\bigoplus_{v} \scL_{v}'$ is a tilting object.
\end{proof}

\begin{lemma} \label{lm:restriction}
Let $s$ be the source of the large hexagon
corresponding to a special representation of $A$ and
$t$ be the target of ``multiplication by $z$''
into the adjacent large hexagon.
Then we have an exact sequence
$$
0 \to \scL_t^\vee \to \scL_s^\vee \to \scL_s^\vee|_Y \to 0.
$$
\end{lemma}

\begin{proof}
Since $\scM'$ is the moduli of representations of $\Gamma'$
by Proposition \ref{prop:complement},
the restriction of the map $\scL_s \to \scL_t$ to $\scM'$ is an isomorphism.
Then the assertion follows from Lemma \ref{lm:reduced}
\end{proof}

\section{Preservation of surjectivity: the general case}
 \label{sc:surj-general}

We use the same notation as in Section \ref{sc:tilting-general}.
In particular,
the quiver $\Gamma'$ is obtained from $\Gamma$
by inverting some of the arrows.

\begin{proposition}\label{prop:surj-general}
Assume that both $\bigoplus \scL_v$ and $\bigoplus \scL'_{v}$
are tilting objects.
Then the map $\bC \Gamma \to \End(\bigoplus \scL_v)$ is surjective
if and only if so is $\bC \Gamma' \to \End(\bigoplus \scL'_{v})$.
\end{proposition}

\begin{proof}
Take a pair $(v, w)$ of vertices of $\Gamma$ and
consider the following commutative diagram
$$
\begin{CD}
 0 @>>> e_v \bC\Gamma e_w 
   @>\gamma>> e_{v} \bC\Gamma' e_{w}
   @>\delta>> Q
   @>>> 0 \\
  @. @VfVV @VgVV @VkVV @. \\
 0 @>>> \Hom(\scL_v, \scL_w)
   @>\alpha>> \Hom(\scL'_{v}, \scL'_{w})
   @>\beta>> H^1_Y(\scL_v^{\vee}\otimes \scL_w)
   @>>> 0
\end{CD}
$$
where $Q$ is defined as the cokernel of $\gamma$.
The second row is exact by \eqref{equation:long-exact} and our assumption.
Moreover, $f$ and $g$ are injective by consistency and
hence the first row is also exact.
The map $k$ is defined so that the diagram is commutative,
and it suffices to show that $k$ is an isomorphism.

In the proof of the surjectivity of
\eqref{equation:restriction} ($=\beta$),
we show that $\beta \circ g$ is surjective and
hence $k$ is surjective.
To see that $k$ is injective,
consider the following commutative diagram
$$
\begin{CD}
 0 @>>> e_v \bC\Gamma e_w
   @>\gamma>> e_{v} \bC\Gamma' e_{w}
   @>\delta>> Q
   @>>> 0 \\
  @. @VVV @ViVV @VjVV @. \\
 0 @>>> e_{\phi_\frakc(v)} \bC\Lambda e_{\phi_\frakc(w)}
   @>>> e_{\phi_\frakc(v)} \bC\Lambda' e_{\phi_\frakc(w)}
   @>\delta'>> Q'
   @>>> 0 \\
  @. @V{f'}VV @V{g'}VV @V{k'}VV @. \\
 0 @>>> \Hom(\scR_{\phi_\frakc(v)}, \scR_{\phi_\frakc(w)})
   @>>> \Hom(\scR'_{\phi_\frakc(v)}, \scR'_{\phi_\frakc(w)})
   @>>> H^1_Y(\scR_{\phi_\frakc(v)}^{\vee}\otimes \scR_{\phi_\frakc(w)})
   @>>> 0
\end{CD}
$$
where $\Lambda$ is the McKay quiver whose vertices are large hexagons.
Here $k'$ is an isomorphism since $f'$ and $g'$ are isomorphisms.

By Lemma \ref{lemma:cancel},
any path in $\bC\Gamma' \setminus \bC \Gamma$
is equivalent to a path that contains an inverse arrow
in the intersection of the two zigzag paths
(corresponding to ``multiplication by $z^{-1}$'')
but not arrows in the corner perfect matching $D$.
This implies that $Q$ (resp.~$Q'$) is isomorphic
to the subspace of $e_{v} \bC\Gamma' e_{w}$
(resp.~$e_{\phi_\frakc(v)'} \bC\Lambda' e_{\phi_\frakc(w)'}$)
spanned by (the classes of) paths that contain inverse arrows
but not arrows contained in $D$.
Therefore, the injectivity of $j$ is reduced to the injectivity of $i$,
which follows from Lemma \ref{prop:inj}.
Now $H^1_Y(\scR_{\phi_\frakc(v)}^{\vee}\otimes \scR_{\phi_\frakc(w)})$ coincides
with $H^1_Y(\scL_v^{\vee}\otimes \scL_w)$
and $k=k'\circ j$ is injective.
\end{proof}

\section{Proof of the derived equivalence}
 \label{sc:der_equiv}

We prove Theorem \ref{th:main} in this section.
Let $G$ be a consistent dimer model.
Since any lattice polygon $\Delta$ can be turned into a triangle
with unit area
by successively removing corners,
one can find a sequence
$$
 G = G_0 \mapsto G_1 \mapsto \cdots \mapsto G_k
$$
of consistent dimer models,
where each step is given by the operation
in Theorem \ref{th:removal},
and the characteristic polygon of $G_k$
is the triangle with unit area.

The dimer model $G_k$ is determined uniquely
by its characteristic polygon
by Proposition \ref{prop:triangle}.
The corresponding quiver is
the McKay quiver for the trivial group, and
the path algebra is isomorphic
to the polynomial algebra in three variables.
In this case,
the moduli space is the affine space and
the tautological bundle is the trivial line bundle,
so that the conditions (\bfT)+(\bfE) are clearly satisfied.

Assume the existence of a derived-equivalence
$$
 \Phi(-)
   = \bR \Gamma \lb \lb \bigoplus_v \scL_v \rb \otimes - \rb
   : D^b \coh \scM_{i,\theta} \to D^b \module \bC \Gamma_i
$$
for some $i > 0$
between the quiver $\Gamma_i$
associated with the dimer model $G_i$
and the moduli space $\scM_{i, \theta}$
of $\theta$-stable representations
of $\Gamma_i$ for some generic $\theta$.
Then we change the stability parameter
to the one described in Proposition \ref{prop:corner}.
This preserves the conditions (\bfT)+(\bfE)
by \cite[Theorem 1.1]{Ishii-Ueda_DMCR}.

Then we use the `if' part of Theorem \ref{th:induction}
to show that conditions (\bfT)+(\bfE) hold
for $G_{i-1}$ for some generic stability parameter.

By repeating this process,
we show that the conditions (\bfT)+(\bfE) hold
for $G$ with any generic stability parameter,
and Theorem \ref{th:main} is proved.

\bibliographystyle{amsalpha}
\bibliography{bibs}

\noindent
Akira Ishii

Department of Mathematics,
Graduate School of Science,
Hiroshima University,
1-3-1 Kagamiyama,
Higashi-Hiroshima,
739-8526,
Japan

{\em e-mail address}\ : \ akira@math.sci.hiroshima-u.ac.jp

\ \\

\noindent
Kazushi Ueda

Department of Mathematics,
Graduate School of Science,
Osaka University,
Machikaneyama 1-1,
Toyonaka,
Osaka,
560-0043,
Japan.

{\em e-mail address}\ : \  kazushi@math.sci.osaka-u.ac.jp

\end{document}